\documentclass[a4paper,11pt,pdf]{amsart}

\usepackage{enumerate, amsmath, amsfonts, amssymb, amsthm, thmtools, wasysym, graphics, graphicx, xcolor, frcursive,comment,bbm}

\usepackage{etex}

\definecolor{darkblue}{rgb}{0.0,0,0.7} 
\newcommand{\darkblue}{\color{darkblue}} 
\definecolor{darkred}{rgb}{0.7,0,0} 
\newcommand{\defn}[1]{\emph{\darkblue #1}} 

\usepackage[all]{xy}
\usepackage[T1]{fontenc}

\usepackage[colorinlistoftodos]{todonotes}

\def\W{{\mathfrak{S}_{n+1}}}

\def\Pc{\mathcal{P}_c}
\def\NC{{\sf{NC}}}
\def\NN{{\sf{NN}}}
\def\Nest{{\sf{Nest}}}
\def\Pol{{\sf{Pol}}}
\def\Krew{{\sf{Krew}}}
\def\KrewNN{{\sf{Krew}}}
\def\J{{J}}
\def\bijD{{\mathcal{D}}}
\def\bijDD{{\mathbf{D}}}
\def\bijVV{{\mathbf{V}}}
\def\bijV{{\mathcal{V}}}
\def\Initial{{\sf{i}}}
\def\Final{{\sf{f}}}
\def\nns{{p}}
\def\ncs{{\pi}}
\def\hgt{{\sf{ht}}}
\def\sl{{\sf{diag}}}
\def\sll{{\sf{vert}}}

\def\ls{{\ell_{\mathcal{S}}}}
\def\T{{\mathcal{T}}}
\def\lt{{\ell_{\T}}}
\def\leqt{{\leq_{\T}}}
\def\leqs{{\leq_{\mathcal{S}}}}
\def\Vn{{\sf{V}}_n}
\def\Dyck{{{\sf{Dyck}}}}

\newtheorem{theorem}{Theorem}[section]
\newtheorem{proposition}[theorem]{Proposition}
\newtheorem{lemma}[theorem]{Lemma}
\newtheorem{corollary}[theorem]{Corollary}

\newtheorem*{notation}{Notation}

\theoremstyle{definition}
\newtheorem{definition}[theorem]{Definition}
\newtheorem{remark}[theorem]{Remark}
\newtheorem{example}[theorem]{Example}

\usepackage{graphicx}                  
\usepackage{pstricks,pst-plot,pst-text,pst-tree,pst-eps,pst-fill,pst-node,pst-math}
\usepackage{setspace}

\newcommand{\ts}{\textsuperscript}

\title{Noncrossing partitions and Bruhat order}
\author{Thomas Gobet}
\address{TU Kaiserslautern, Fachbereich Mathematik, Postfach 3049, 67653 Kaiserslautern, Germany.}
\email{gobet@mathematik.uni-kl.de} 
\author{Nathan Williams}
\address{LaCIM, Universit\'e du Qu\'ebec \`a Montr\'eal \\
201, Pr\'esident-Kennedy, 4\`eme \'etage \\
Montr\'eal (Qu\'ebec) H2X 3Y7, Canada.}
\email{nathan.f.williams@gmail.com}

\begin{document}

\maketitle

\begin{abstract}
We prove that the restriction of Bruhat order to noncrossing partitions in type $A_n$ for the Coxeter element $c=s_1s_2 \cdots s_n$ forms a distributive lattice isomorphic to the order ideals of the root poset ordered by inclusion.  Motivated by the change-of-basis from the graphical basis of the Temperley-Lieb algebra to the image of the simple elements of the dual braid monoid, we extend this bijection to other Coxeter elements using certain canonical factorizations.  In particular, we give new bijections---fixing the set of reflections---between noncrossing partitions associated to distinct Coxeter elements.
\end{abstract}

\tableofcontents

\section{Introduction}

Fix the Coxeter group of type $\mathfrak{S}_{n+1}$, its identity $e$, the linear Coxeter element $c=(1,2,\ldots,n+1)$, and its positive root poset $\Phi^+(\mathfrak{S}_{n+1})$.  Let $\NC(\W,c)$ be elements in the absolute order interval $[e,c]$---that is, the set of \defn{noncrossing partitions} with respect to $c$.  Let $\NN(\mathfrak{S}_{n+1})$ be the set of order ideals in $\Phi^+(\mathfrak{S}_{n+1})$, commonly known as the \defn{nonnesting partitions}.

Our main theorem is a relationship between these noncrossing and nonnesting partitions.

\begin{theorem}
\label{thm:thm1}
	The restriction of Bruhat order to $\NC(\W,c)$ is isomorphic to the distributive lattice given by ordering the elements of $\NN(\mathfrak{S}_{n+1})$ by inclusion.
\end{theorem}

We remark that a similar construction is given as Example $6.4$ in \cite{Armst}, though this uses a different order restricted to a different set of elements.

Our motivation for the study of Bruhat order on noncrossing partitions comes from Temperley-Lieb algebras.  The classical Temperley-Lieb algebras have---among others---the following two bases: the usual one, indexed by fully-commutative permutations of the symmetric group (the diagram basis), and a less-well-known one, indexed by noncrossing partitions.  The \defn{simple elements} of the Birman-Ko-Lee braid monoid are canonical lifts of noncrossing partitions to the braid group~\cite{BKL}.  The second basis, originally described by Zinno~\cite{Z} (see also \cite{Lee}), is given explicitly as the image of these simple elements in the Temperley-Lieb algebra.  It is natural to consider the change-of-basis matrix, which---for a certain ordering on the simple elements---was shown by Zinno to be upper triangular, with invertible coefficients on the diagonal.  Zinno's ordering was refined in~\cite{GobTh,GobBase} to coincide with linear extensions of the restriction of Bruhat order to $\NC(\W,c)$.

We prove Theorem~\ref{thm:thm1} over Sections~\ref{sec:nc_partitions},~\ref{sec:nn_partitions},~\ref{sec:bijections},~\ref{sec:vectors}, and~\ref{sec:bruhat}.  In Sections~\ref{sec:nc_partitions} and~\ref{sec:nn_partitions}, we recall the combinatorial and algebraic definitions of noncrossing and nonnesting partitions.  In Section~\ref{sec:bijections}, we present three equivalent bijections between noncrossing and nonnesting partitions, and prove their equivalence.  Two of these bijections are of a similar flavor as a product over certain labelings over the root poset; both end up defining (reduced) factorizations of a noncrossing partition into simple reflections.  We use these labelings to define in Section~\ref{sec:vectors} two different collections of vectors, both counted by the Catalan numbers.  These vectors are obtained by counting the number of occurrences of simple reflections in the two factorizations. We characterize the vectors in both cases. To our knowledge, one of these collections of vectors is a new Catalan object---it does not appear in~\cite{stanleyadd}.  In Section~\ref{sec:bruhat}, we prove Theorem~\ref{thm:thm1} in two ways by showing that the bijections from Section~\ref{sec:bijections} define a poset isomorphism.


Bessis' dual braid monoid is a generalization of the Birman-Ko-Lee braid monoid to arbitrary Coxeter elements and all finite Coxeter groups \cite{Dual}---the Birman-Ko-Lee monoid corresponds to the Coxeter group of type $A_n$ and $c=s_1s_2\cdots s_n$.  The dual braid monoid is a Garside monoid; as such, it has a set of simple elements. In Section~\ref{sec:extensions}, we consider the case of noncrossing partitions $\NC(\W,c')$ associated to an arbitrary Coxeter element $c'$; they can be lifted to the simple elements of the corresponding dual braid monoid.  While the lattice property of the Bruhat order restricted to $\NC(\W,c')$ fails in general, one can still map the simple elements of the dual braid monoid (isomorphic to the Birman-Ko-Lee braid monoid, but realized differently as a submonoid of the braid group) to the Temperley-Lieb algebra, and it turns out that images of the simple elements still yield a basis.  The restriction of Bruhat order to noncrossing partitions is not necessarily the ordering required to diagonalize the change-of-basis matrix for other Coxeter elements, and we adapt Theorem~\ref{thm:thm1} accordingly.  We succeed in generalizing the theorem using only combinatorial and Coxeter-theoretic notions, without any explicit reference to braid groups and monoids or Temperley-Lieb algebras.  The key is to define (not necessarily reduced) factorizations using new bijections \emph{not} obtained by conjugation---they preserve the reflection length but fix the set of reflections and also preserve the support of the elements---between noncrossing partitions associated to distinct Coxeter elements.  We believe that similar bijections with the same properties should exist in other types.  The poset obtained is again isomorphic to the lattice of order ideals in the root poset. These factorizations are used for the study of the above mentioned bases of Temperley-Lieb algebras in the first author's thesis~\cite{GobTh}, in particular the new ordering gives triangularity of the change of basis matrix mentioned above.  


We conclude with a brief discussion of a bijection between type $B$ noncrossing and nonnesting partitions.

\textbf{Acknowledgments}. The first author thanks Philippe Nadeau and the Institut Camille Jordan for an invitation in Lyon in February 2014 and Fr\'ed\'eic Chapoton for pointing out that the lattices for types $A_2$ and $A_3$ where the same as the lattices of Dyck paths ordered by inclusion, among other enriching discussions with both of them. Thanks also go to Fran\c{c}ois Digne for reading preliminary versions of parts of this paper. The second author thanks Drew Armstrong for showing him the construction given in Figure \ref{fig:krewcomp}.

\section{Noncrossing Partitions}
\label{sec:nc_partitions}
In this section, we introduce the noncrossing partitions as combinatorial objects, and as objects attached to a Coxeter group of type $A$.

\subsection{Combinatorial Noncrossing Partitions}
Let $[n+1]:=\{1,2,\ldots,n+1\}$.  A \defn{set partition} of $[n+1]$ is an unordered collection $\pi$ of nonempty subsets of $[n+1]$ with empty pairwise intersection whose union is all of $[n+1]$; these subsets are called \defn{blocks}.  Given a set partition $\pi$, a \defn{bump} is a pair $(i_1,i_2)$ with $i_1<i_2$ in the same block of $\pi$ such that there is no $j$ in that block with $i_1<j<i_2$.

\begin{definition}
\label{def:ncpartitionstypea}
A \defn{noncrossing partition} $\ncs$ of $[n+1]$ is a set partition with the condition that
		if $(i_1,i_2)$, $(j_1,j_2)$ are two distinct bumps in $\ncs$, then it is not the case that $i_1<j_1<i_2<j_2$.
\end{definition}

The \defn{graphical representation} of a noncrossing partition $\ncs$ is the set of convex hulls of the blocks of $\ncs$ when drawn around a regular $(n+1)$-gon with vertices labeled with $[n+1]$ in clockwise order.  The definition above is equivalent to non-intersection of the hulls.  Ordering the noncrossing partitions by refinement yields the \defn{noncrossing partition lattice}~\cite{kreweras1972partitions}.  This lattice is drawn for $n=4$ in Figure~\ref{fig:noncrossing41}.

We write $\Pol(\ncs)$ for the set of polygons occuring in the geometric representation of $\ncs$. A polygon $P\in\Pol(\ncs)$ is given by an ordered sequence of indices $P=[i_1 i_2 \cdots i_k],$ where $i_j$ are the index the vertices of $P$ and $i_1<i_2<\dots<i_k$.
 We say that $i_1$ is an \defn{initial} index for $P$ and $i_k$ a \defn{terminal} one.  For example, in Figure~\ref{figure:nonc} the polygons correspond to the sequences $P_1=[235]$ and $P_2=[16]$; the initial index for $P_1$ is $2$ and its final index is $5$, while the initial index for $P_2$ is $1$ and its final index is $6$.


\begin{figure}[htbp]
 \includegraphics[height=3in]{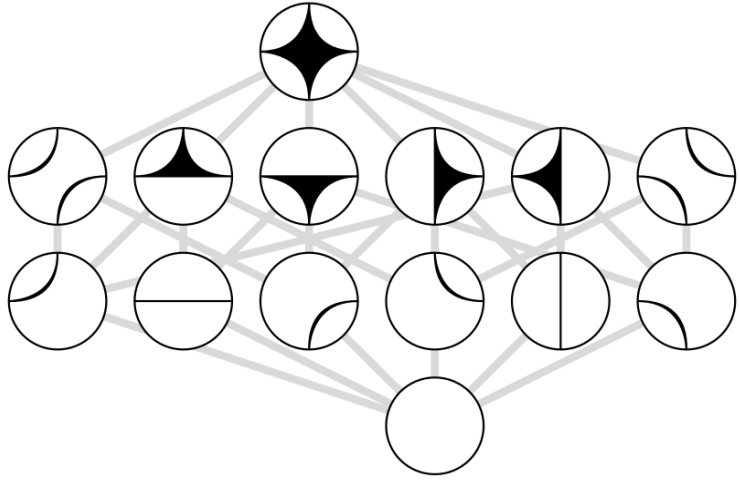}
\caption{The noncrossing partition lattice $\NC(\mathfrak{S}_{4},c)$.}
\label{fig:noncrossing41}
\end{figure}

\begin{figure}[htbp]
\begin{center}
\begin{tabular}{ccc}
& \begin{pspicture}(0,0)(6,3)
\pscircle(3,1.5){1.5}
\psdots(4.5,1.5)(3.75,2.79)(3.75,.21)(2.25,.21)(2.25,2.79)(1.5,1.5)
\pspolygon[linecolor=orange, fillstyle=solid, fillcolor=orange](3.75,2.79)(4.5,1.5)(2.25,.21)
\psline[linecolor=orange](2.25,2.79)(1.5,1.5)
\rput(3.75,.21){$\bullet$}
\rput(4.5,1.5){$\bullet$}
\rput(3.75,2.79){$\bullet$}
\rput(2.25,.21){$\bullet$}
\rput(2.25,2.79){$\bullet$}
\rput(1.5,1.5){$\bullet$}
\rput(4.05,2.92){\textrm{{\footnotesize \textbf{2}}}}
\rput(4.8,1.5){\textrm{{\footnotesize \textbf{3}}}}
\rput(4,.03){\textrm{{\footnotesize \textbf{4}}}}
\rput(2,.03){\textrm{{\footnotesize \textbf{5}}}}
\rput(1.17,1.5){\textrm{{\footnotesize \textbf{6}}}}
\rput(1.88,2.92){\textrm{{\footnotesize \textbf{1}}}}
\end{pspicture} & \\
\end{tabular}
\end{center}
\caption{The geometric representation of the noncrossing partition $\ncs=(1,6)(2,3,5) \in \NC(\mathfrak{S}_{6},c)$.}
\label{figure:nonc}
\end{figure}
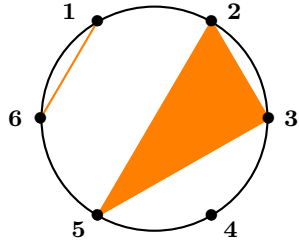

\subsection{Algebraic Noncrossing Partitions}
\label{sec:nc_algebraic}
Let $\mathfrak{S}_{n+1}$ be the symmetric group on $n+1$ letters, $\mathcal{S}$ the set of \defn{simple reflections} $\{s_i:=(i,i+1)\}_{i=1}^n$, and $\T:=\{w s w^{-1} : w \in \mathfrak{S}_{n+1}, s \in \mathcal{S} \}$ the set of all \defn{reflections}.  The \defn{length} of an element $w \in \W$ is the minimal number $\ls(w)$ of simple reflections required to write $w = s_{i_1}\cdots s_{i_{\ls(w)}}$. The \defn{support} of an element $w\in \W$, written $\mathrm{supp}(w)$, is the set $\{j\in[n+1]~|~w(j)\neq j\}$. 
 Length induces the weak order $\leqs$ on $\mathfrak{S}_{n+1}$, which is the partial order \[u\leqs v \text{ iff } \ls(u)+\ls(u^{-1} v)=\ls(v).\]  We write elements of the Coxeter group as $w$ and an $\mathcal{S}$-word for $w$ as $\mathbf{w}$.  Thus, if $w, w'\in\mathfrak{S}_{n+1}$, then their product is written $ww'$; if $\mathbf{w},\mathbf{w'}$ are $\mathcal{S}$-words, $\mathbf{ww'}$ is their concatenation.

Similarly, the \defn{reflection length} for an element $w \in \mathfrak{S}_{n+1}$ is the minimal number $\lt(w)$ of reflections required to write $w = t_1 \cdots t_{\lt(w)}$.  The reflection length induces the \defn{absolute order} $\leqt$ on $\mathfrak{S}_{n+1}$, which is the partial order \[ u\leqt v \text{ iff } \lt(u)+\lt(u^{-1} v)=\lt(v).\]  A (standard) \defn{Coxeter element} is a product $\prod_{i=1}^n s_{\pi(i)}$ for some permutation $\pi \in \mathfrak{S}_n$.  Until Section~\ref{sec:extensions}, we will work only with the Coxeter element $c:=s_1 s_2\cdots s_n \in \mathfrak{S}_{n+1}$ (so that $c$ is the long cycle $(1,2,\ldots,n+1)$).  We let \[\NC(\W,c):=[e,c]_{\T}\] be the interval of elements below $c$ in absolute order.

	It is well-known that mapping a permutation $\ncs \in \NC(\W,c)$ to the set partition whose blocks are given by the cycles of $\ncs$ is an isomorphism of posets between $\NC(\W,c)$ and the combinatorial noncrossing partition lattice.  To make this explicit, define a reduced $\mathcal{S}$-word for the transposition $(j, k)$ with $j<k$ by the word \[\mathbf{s_{[j, k]}}:=(s_{k-1} s_{k-2}\cdots s_{j+1}) s_j (s_{j+1}\cdots s_{k-2} s_{k-1}).\]  We call $\mathbf{s_{[j, k]}}$ a \defn{syllable}.
	
	Now order the set of polygons $P_1,\dots P_r\in\Pol(\ncs)$ in ascending order of the maximal integer in each polygon.  A polygon $P_i=[i_1 i_2 \cdots i_k]$ in the geometric representation of $\ncs \in \NC(\W,c)$ represents the cycle $\ncs_i =(i_1, i_2, \dots, i_k) \in \NC(\W,c)$. We define the $\mathcal{S}$-word $\mathbf{\ncs_i}$ for $\ncs_i$ by the concatenation \[\mathbf{\ncs_i}:=\mathbf{s_{[i_1,i_2]}s_{[i_2,i_3]}\cdots s_{[i_{k-1},i_k]}}.\]  Then $\ncs=\ncs_1 \ncs_2\cdots \ncs_r,$ and the concatenation \[\mathbf{\ncs}:=\mathbf{\ncs_1 \ncs_2 \cdots \ncs_r}\] is a reduced $\mathcal{S}$-word for $\ncs$, so that $\sum_{i=1}^n \ell_{\mathcal{S}}(\ncs_i)=\ell_{\mathcal{S}}(\ncs)$.  We write this reduced word $m_{\ncs}$ and call it the \defn{standard form} of $\ncs$.
\begin{example}
Consider the noncrossing partition $\ncs=(1,6)(2,3,5)$ from figure \ref{figure:nonc}. Then $$m_{\pi}=\mathbf{s_{[1,2]}}\mathbf{s_{[2,4]}}\mathbf{s_{[1,6]}}=(s_2 s_4s_3s_4)(s_5s_4s_3s_2s_1s_2s_3s_4s_5).$$ 
\end{example}


\begin{remark}
We chose to write the cycle $(i_1< \cdots< i_k)$ as the reduced $\mathcal{T}$-word \[(i_1,i_2)(i_2,i_3)\cdots (i_{k-1},i_k).\]  This decomposition is a special case of Theorem 3.5 in~\cite{ABW}---as in Example 3.3 of~\cite{ABW}, when the cycles of $x \in \NC(\W,c)$ are ordered correctly, this choice gives a factorization of $x$ into a product of reflections that occur in lexicographic order.
\end{remark}

We now define the Kreweras complement on noncrossing partitions, which will be used in Section~\ref{sec:nn_diagonal}, also introducing noncrossing matchings for use in Section~\ref{sec:dyck}.

\begin{definition}
Define $\Krew_c: \W \to \W$ by $w \mapsto w^{-1}c$. The restriction to the noncrossing partitions is a bijection, $\Krew_c: \NC(\W,c) \to \NC(\W,c)$, called the \defn{Kreweras complement}.
\end{definition}

The Kreweras complement $\Krew_c$ is an anti-isomorphism of the poset $\NC(\W,c)$. It may be visualized as rotation using noncrossing matchings---for this, it is probably most efficient to refer the reader to the graphical construction of $\Krew_c$ in Figure~\ref{fig:krewcomp} that ``thickens'' a noncrossing partition on $[n+1]$ to a \defn{noncrossing matching} on $[2(n+1)]$.

\begin{figure}[htbp]
$$\xymatrix{
\raisebox{-0.5\height}{\includegraphics[height=1.1in]{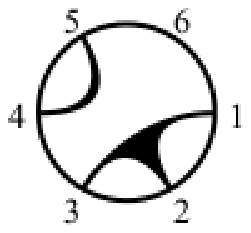}}	\ar@{->}[r]^{\Krew} \ar@{->}[d] \hspace{5pt} & \hspace{5pt} \raisebox{-0.5\height}{\includegraphics[height=1.1in]{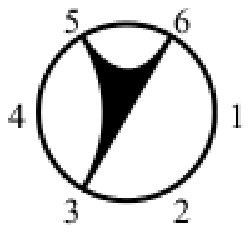}} \\
\raisebox{-0.5\height}{\includegraphics[height=1.1in]{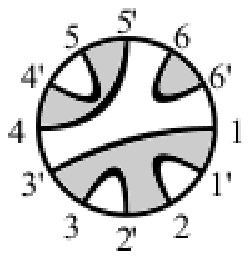}}	\ar@{->}[r]_{\text{Rotation}} \hspace{5pt} & \hspace{5pt} \raisebox{-0.5\height}{\includegraphics[height=1.1in]{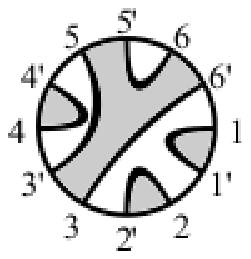}} \ar@{->}[u] \\
}$$
\caption[Interpreting Kreweras complementation as rotation.]{Interpreting Kreweras complementation as rotation by ``thickening'' a noncrossing partition to a noncrossing matching.}
\label{fig:krewcomp}
\end{figure}




\section{Nonnesting Partitions}
\label{sec:nn_partitions}
In this section, we define nonnesting partitions as purely combinatorial objects and as objects associated to root systems of type $A$.

\subsection{Combinatorial Nonnesting Partitions}
\label{sec:comb_nn}
\begin{definition}
\label{def:nnpartitionstypea}
A \defn{nonnesting partition} of $[n+1]$ is a set partition $\nns$ with the condition that
	if $(i_1,i_2)$, $(j_1,j_2)$ are two distinct bumps in $\nns$, then it is not the case that $i_1<j_1<j_2<i_2$. We write $\NN(\W)$ for the set of all nonnesting partitions.
\end{definition}

We can represent a partition $\mu$ graphically by placing $n$ labeled collinear points and then drawing an arc from $i_1$ to $i_2$ when $(i_1,i_2)$ is a bump in $\mu$.  We emphasize that the names \textit{noncrossing} and \textit{nonnesting} come from this linear model: the noncrossing partitions avoid pairs of arcs that \textit{cross}, and the nonnesting partitions avoid pairs of arcs that \textit{nest}.  For example, the noncrossing partition $14|23$ and the nonnesting partition $13|24$ are represented in Figure~\ref{fig:partition14231324}.  A larger example of a noncrossing partition is given in Figure~\ref{figure:diagramme}.

\begin{figure}[h!]
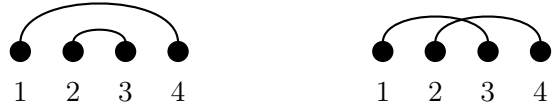

\begin{center}
\begin{psmatrix}[colsep=0.5,rowsep=0.2]
~ & ~ & ~ & ~ & ~ & 
~ & ~ & ~ & ~ & ~ & ~ & ~\\
~ & ~ & ~ & ~ & ~ & 
~ & ~ & ~ & ~ & ~ & ~ & ~\\
\pscircle*{0.14} & \pscircle*{0.14} & \pscircle*{0.14} & \pscircle*{0.14} & & & & & \pscircle*{0.14} & \pscircle*{0.14} & \pscircle*{0.14} & \pscircle*{0.14}\\
1 & 2 & 3 & 4 & & 
& & & 1 & 2 & 3 & 4\\
\end{psmatrix}
\ncarc[arcangle=90]{3,1}{3,4}
\ncarc[arcangle=90]{3,2}{3,3}
\ncarc[arcangle=90]{3,9}{3,11}
\ncarc[arcangle=90]{3,10}{3,12}
\end{center}
\caption[The noncrossing and nonnesting partitions $\ncs=14|23$ and $\nns=13|24$.]{The noncrossing and nonnesting partitions $\ncs=14|23$ and $\nns=13|24$.}
\label{fig:partition14231324}
\end{figure}

For noncrossing partitions, the only difference between the drawing above and the geometrical representation on a circle is that the longest edge of each polygon is not drawn.

\begin{figure}[h!]
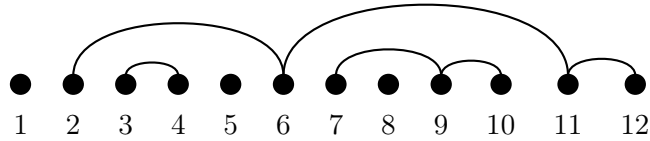

\begin{center}
\begin{psmatrix}[colsep=0.5,rowsep=0.2]
~ & ~ & ~ & ~ & ~ & 
~ & ~ & ~ & ~ & ~ & ~ & ~\\
~ & ~ & ~ & ~ & ~ & 
~ & ~ & ~ & ~ & ~ & ~ & ~\\
\pscircle*{0.14} & \pscircle*{0.14} & \pscircle*{0.14} & \pscircle*{0.14} & \pscircle*{0.14} & 
\pscircle*{0.14} & \pscircle*{0.14} & \pscircle*{0.14} & \pscircle*{0.14} & \pscircle*{0.14} & \pscircle*{0.14} & \pscircle*{0.14}\\
1 & 2 & 3 & 4 & 5 & 
6 & 7 & 8 & 9 & 10 & 11 & 12\\
\end{psmatrix}
\ncarc[arcangle=90]{3,2}{3,6}
\ncarc[arcangle=90]{3,3}{3,4}
\ncarc[arcangle=90]{3,6}{3,11}
\ncarc[arcangle=90]{3,7}{3,9}
\ncarc[arcangle=90]{3,9}{3,10}
\ncarc[arcangle=90]{3,11}{3,12}
\end{center}
\caption{A diagram for $n=11$ representing the noncrossing partition $x=(2,6,11,12)(3,4)(7,9,10)$.} 
\label{figure:diagramme}
\end{figure}


\subsection{Algebraic Nonnesting Partitions}
We write ~$\Phi^+$ for the set of \defn{positive roots} of $\W$, that is, \[\Phi^+:=\{e_i-e_j : 1\leq i < j \leq n+1\},\] where $\{e_i\}_{i=1}^{n+1}$ is a basis of $\mathbb{R}^{n+1}$.  The \defn{root poset} is the order on $\Phi^+$ defined by $\alpha \leq \beta$ if and only if $\beta-\alpha \in \mathbb{Z}\Phi^+$.  We abuse notation and also denote the root poset by $\Phi^+$.  We define the \defn{height} $\hgt(\alpha)$ of a positive root $\alpha$ to be its rank in the positive root poset.

  The combinatorial nonnesting partitions of Definition~\ref{def:nnpartitionstypea} are in bijection with antichains in the positive root poset of $\Phi^+$ by mapping a bump $(i,j)$ to the root $e_i-e_j$.  We may complete these antichains to order ideals by sending an antichain $\mathcal{A}$ of a poset $\mathcal{P}$ to the order ideal \[\{p \in \mathcal{P} \hspace{5pt} : \hspace{5pt} p \leq p' \text{ for some } p'\in \mathcal{A}\},\] so that nonnesting partitions naturally form a distributive lattice.

\section{Bijections}
\label{sec:bijections}

In this section, we present three bijections between noncrossing and nonnesting partitions, and then show that they are the same bijection.  The first bijection is purely combinatorial, while the second and third use two different labelings of the root poset to produce factorizations of a noncrossing partition.

\subsection{Dyck Paths}
\label{sec:dyck}
A \defn{Dyck path} is a path from $(0,0)$ to $(2n, 0)$ in $\mathbb{R}^2$ with steps of the form $+(1,1)$ or $+(1,-1)$ which stays above the $x$-axis.  If we plot a Dyck path on top of the positive root poset, it traces out an order ideal of the root poset, establishing a bijection between Dyck paths with $2n$ steps and nonnesting partitions of $\W$.  An example is given in Figure~\ref{fig:nn_eq_dyck}.  Due to the superficiality of this bijection, we will simply associate Dyck paths with $\NN(\W)$.  

\begin{figure}
\psscalebox{0.55}{\begin{pspicture}(-1,0)(20,4)

\pspolygon[linecolor=lightgray, fillstyle=solid, fillcolor=lightgray](-1,0)(2,3)(3,2)(5,4)(9,0)
\psdot(1,1)
\psdot(3,1)
\psdot(5,1)
\psdot(7,1)
\psdot(2,2)
\psdot(4,2)
\psdot(6,2)
\psdot(3,3)
\psdot(5,3)
\psdot(4,4)
\psline(1,1)(2,2)
\psline(2,2)(3,3)
\psline(3,3)(4,4)
\psline(4,4)(5,3)
\psline(5,3)(6,2)
\psline(6,2)(7,1)
\psline(2,2)(3,1)
\psline(3,1)(4,2)
\psline(4,2)(5,1)
\psline(5,1)(6,2)
\psline(3,3)(4,2)
\psline(4,2)(5,3)

\pspolygon[linecolor=lightgray, fillstyle=solid, fillcolor=lightgray](10,0)(13,3)(14,2)(16,4)(20,0)
\psdot(12,1)
\psdot(14,1)
\psdot(16,1)
\psdot(18,1)
\psdot(13,2)
\psdot(15,2)
\psdot(17,2)
\psdot(14,3)
\psdot(16,3)
\psdot(15,4)
\psline(12,1)(13,2)
\psline(13,2)(14,3)
\psline(14,3)(15,4)
\psline(15,4)(16,3)
\psline(16,3)(17,2)
\psline(17,2)(18,1)
\psline(13,2)(14,1)
\psline(14,1)(15,2)
\psline(15,2)(16,1)
\psline(16,1)(17,2)
\psline(14,3)(15,2)
\psline(15,2)(16,3)

\psline[linewidth=2.5pt](10,0)(11,1)(12,2)(13,3)(14,2)(15,3)(16,4)(17,3)(18,2)(19,1)(20,0)
\psdot[dotsize=8pt, dotstyle=o](10,0)
\psdot[dotsize=8pt, dotstyle=o](11,1)
\psdot[dotsize=8pt, dotstyle=o](12,2)
\psdot[dotsize=8pt, dotstyle=o](13,3)
\psdot[dotsize=8pt, dotstyle=o](14,2)
\psdot[dotsize=8pt, dotstyle=o](15,3)
\psdot[dotsize=8pt, dotstyle=o](16,4)
\psdot[dotsize=8pt, dotstyle=o](17,3)
\psdot[dotsize=8pt, dotstyle=o](18,2)
\psdot[dotsize=8pt, dotstyle=o](19,1)
\psdot[dotsize=8pt, dotstyle=o](20,0)

\end{pspicture}}
\caption{Nonnesting partitions are Dyck paths}
\label{fig:nn_eq_dyck}
\end{figure}
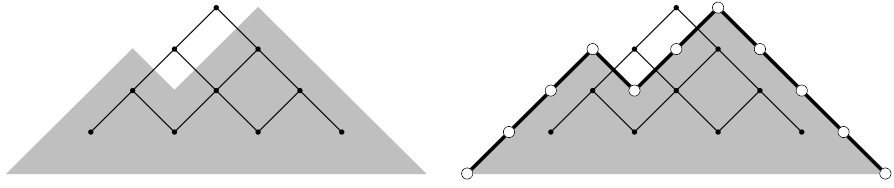

For a noncrossing partition $\ncs$ drawn using the linear representation of Section~\ref{sec:comb_nn}, we associate a Dyck path $\nns_\ncs$ by turning each point in $[n+1]$ into the $(2i-1)$\ts{st} and the $(2i)$\ts{th} steps of $\nns_\ncs$ in the following way:
\begin{itemize}
\item (Step I) If there are no arcs starting or ending at point $i$, then the $(2i-1)$\ts{st} step is up and the $(2i)$\ts{th} is down;
\item (Step II) If there are arcs starting and ending at point $i$, then the $(2i-1)$\ts{st} step is down and the $(2i)$\ts{th} is up;
\item (Step III) If there is a single arc starting at point $i$, then the $(2i-1)$\ts{st} and $(2i)$\ts{th} steps of the Dyck path are both up; and
\item (Step IV) If there is a single arc ending at point $i$, then the $(2i-1)$\ts{st} and $(2i)$\ts{th} steps are both down.
\end{itemize}

As an example, the path associated with the noncrossing partition from Figure~\ref{figure:diagramme} above is given in Figure~\ref{figure:dyck}.
\begin{figure}
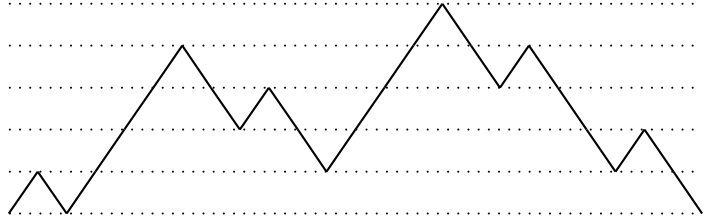

\begin{center}
\begin{psmatrix}[colsep=0.38,rowsep=0.1]

~ & ~ & ~ & ~ & ~ & ~ & ~ & ~ & ~ & ~ & ~ & ~ & ~ & ~ & ~ & ~ & ~ & ~ & ~ & ~ & ~ & ~ & ~ & ~ & ~ \\
~ & ~ & ~ & ~ & ~ & ~ & ~ & ~ & ~ & ~ & ~ & ~ & ~ & ~ & ~ & ~ & ~ & ~ & ~ & ~ & ~ & ~ & ~ & ~ & ~ \\ 
~ & ~ & ~ & ~ & ~ & ~ & ~ & ~ & ~ & ~ & ~ & ~ & ~ & ~ & ~ & ~ & ~ & ~ & ~ & ~ & ~ & ~ & ~ & ~ & ~ \\
~ & ~ & ~ & ~ & ~ & ~ & ~ & ~ & ~ & ~ & ~ & ~ & ~ & ~ & ~ & ~ & ~ & ~ & ~ & ~ & ~ & ~ & ~ & ~ & ~\\
~ & ~ & ~ & ~ & ~ & ~ & ~ & ~ & ~ & ~ & ~ & ~ & ~ & ~ & ~ & ~ & ~ & ~ & ~ & ~ & ~ & ~ & ~ & ~ & ~\\
~ & ~ & ~ & ~ & ~ & ~ & ~ & ~ & ~ & ~ & ~ & ~ & ~ & ~ & ~ & ~ & ~ & ~ & ~ & ~ & ~ & ~ & ~ & ~ & ~\\
~ & ~ & ~ & ~ & ~ & ~ & ~ & ~ & ~ & ~ & ~ & ~ & ~ & ~ & ~ & ~ & ~ & ~ & ~ & ~ & ~ & ~ & ~ & ~ & ~ \\
\end{psmatrix}
\ncline{7,1}{6,2}
\ncline{6,2}{7,3}
\ncline{7,3}{6,4}
\ncline{6,4}{5,5}
\ncline{5,5}{4,6}
\ncline{4,6}{3,7}
\ncline{3,7}{4,8}
\ncline{4,8}{5,9}
\ncline{5,9}{4,10}
\ncline{4,10}{5,11}
\ncline{5,11}{6,12}
\ncline{6,12}{5,13}
\ncline{5,13}{4,14}
\ncline{4,14}{3,15}
\ncline{3,15}{2,16}
\ncline{2,16}{3,17}
\ncline{3,17}{4,18}
\ncline{4,18}{3,19}
\ncline{3,19}{4,20}
\ncline{4,20}{5,21}
\ncline{5,21}{6,22}
\ncline{6,22}{5,23}
\ncline{5,23}{6,24}
\ncline{6,24}{7,25}
\ncline[linestyle = dotted]{7,1}{7,25}
\ncline[linestyle = dotted]{6,1}{6,25}
\ncline[linestyle = dotted]{5,1}{5,25}
\ncline[linestyle = dotted]{4,1}{4,25}
\ncline[linestyle = dotted]{3,1}{3,25}
\ncline[linestyle = dotted]{2,1}{2,25}
\end{center}
\caption{Dyck path for $x=(2,6,11,12)(3,4)(7,9,10)$.} 
\label{figure:dyck}
\end{figure}

\begin{proposition}
\label{prop:nc_dyck}
The assignment $\ncs \mapsto \nns_\ncs$ is a bijection from $\NC(\W,c)$ to $\NN(\W)$ such that $|\ls(\ncs)|=|\nns_\ncs|$.
\end{proposition} 

\begin{proof}
It is immediate to pass between noncrossing matchings and Dyck paths.  The result follows from the bijection between noncrossing matchings and noncrossing partitions illustrated in Figure~\ref{fig:krewcomp}, and Dyck paths and nonnesting partitions illustrated in Figure~\ref{fig:nn_eq_dyck}.  The statement about length will follow from Proposition~\ref{prop:nn_to_nc_vert} and Theorem~\ref{thm:all_equivalent}.
\end{proof}
For $p\in\NN(\W)$, we write $\Dyck(p)$ for the corresponding element of $\NC(\W,c)$ under the inverse of the above bijection.   We remark that there are many similar bijections in the literature; see \cite{stanleyadd}. Notice that in \cite[Section 3]{ferr}, a different bijection between Dyck paths and noncrossing partitions is given.

\medskip

We introduce some additional notation.  For $\ncs \in\NC(\W,c)$ we define its \defn{initial set} $D_\ncs \subset\{1,\dots,n\}$ (resp. \defn{final set} $U_\ncs \subset\{2,\dots,n+1\}$) to be the set containing those integers $m$ for which there exists a polygon $P=[i_1 i_2 \cdots i_k]$ of $\ncs$ and an integer $j\neq k$ (resp. $j\neq 1$) such that $m=i_j$.  That is, $D_\ncs$ (resp. $U_\ncs$) contains all non-terminal (resp. non-initial) indices indexing the vertices of any polygon of $\ncs$.  We abuse notation and write $m\in P$ if there exists $1\leq j\leq k$ with $m=i_j$. It is clear that $|D_\ncs|=|U_\ncs|$.  In Figure \ref{figure:nonc}, $D_\ncs=\{1, 2, 3\}$ and $U_\ncs=\{3,5,6\}$.
  We now characterize the pairs $\{ (D_\ncs,U_\ncs) : \ncs \in \NC(\W,c)\}.$

\begin{proposition}\label{prop:terminalinitial}
 Let $0\leq k<n+1$ and let $\mathcal{I}_k$ be the set of pairs $(D, U)$ where $D=\{d_1, d_2,\dots, d_k\}$, $U=\{e_1, e_2,\dots, e_k\}$ such that
\begin{itemize}
	\item $e_i,d_i\in[n+1]$,
	\item $d_i<d_{i+1}$, $e_i<e_{i+1}$ for each $1\leq i<k$, and
	\item $d_i< e_i$ for each $1\leq i\leq k$.
\end{itemize}
Let $\mathcal{I}:=\bigcup_{k=0}^n \mathcal{I}_k$.  Then the map $\epsilon:\NC(\W,c)\rightarrow\mathcal{I}$, $\ncs\mapsto (D_\ncs,U_\ncs)$ is a bijection.
\end{proposition}
\noindent The objects defined in Proposition~\ref{prop:terminalinitial} are item $(l^6)$ in \cite{stanleyadd}.
\begin{proof}
Given $(D,U) \in \mathcal{I}$, we construct a Dyck path as a sequence of pairs of steps.

\begin{itemize}
\item (Step I) If $i \not \in U$ and $i \not \in D$, then the $(2i-1)$\ts{st} step is up and the $(2i)$\ts{th} is down;
\item (Step II) If $i \in U$ and $i \in D$, then the $(2i-1)$\ts{st} step is down and the $(2i)$\ts{th} is up;
\item (Step III) If $i \in D$ and $i \not \in U$, then the $(2i-1)$\ts{st} and $(2i)$\ts{th} steps of the Dyck path are both up; and
\item (Step IV) If $i \in U$ and $i \not \in D$, then the $(2i-1)$\ts{st} and $(2i+1)$\ts{th} steps are both down.
\end{itemize}

This is clearly reversible, giving a bijection between Dyck paths and $\mathcal{I}$.  Using Proposition~\ref{prop:nc_dyck}, we conclude the result.
\end{proof}

\begin{remark}
The bijection $\NC(\W,c) \overset{\sim}{\rightarrow} \mathcal{I}$ given above can be used to relate $\NC(\W,c)$ to the \defn{fully commutative} elements of $\mathfrak{S}_{n+1}$---that is, those elements for which any two $\mathcal{S}$-reduced expressions are equal up to commutations (these are also known as the $321$-avoiding permutations, see~\cite{stem}).  Each fully commutative element $w$ can be written uniquely in the form \[(s_{i_1} s_{i_1-1}\cdots s_{j_1})(s_{i_2} s_{i_2-1}\cdots s_{j_2})\cdots (s_{i_k} s_{i_k-1}\cdots s_{j_k}),\]
where each index lies in $\{1,\dots,n\}$, $i_1<i_2<\dots<i_k$, $j_1<j_2<\dots<j_k$ and $j_m\leq i_m$ for each $1\leq m\leq k$.  Conversely, any element written in this form is fully commutative.

Then there is a bijection between the fully commutative elements and $\mathcal{I}$ given by \[w\mapsto \left(\{j_1,\dots, j_k\},\{i_1+1,\dots, i_k+1\} \right).\]
\end{remark}

\subsection{Vertical Labeling}
\label{sec:nn_vertical}

In this section we give a bijection from nonnesting partitions to noncrossing partitions as a product over certain labels of the elements of the nonnesting partition.

We first define a labeling of the positive roots $\Phi^+$ by \[\sll(e_i-e_j) := \left\lceil \frac{j+i-1}{2} \right\rceil.\]  For example, $\Phi^+$ for $\mathfrak{S}_4$ is labeled as in Figure \ref{fig:vert_label}.

\begin{center}
\begin{figure}[h!]
\begin{pspicture}(0,0)(3,1.5)
\rput(0,0){1}
\rput(1,0){2}
\rput(2,0){3}
\rput(3,0){4}
\rput(0.5,0.5){2}
\rput(1.5,0.5){3}
\rput(2.5,0.5){4}
\rput(1,1){2}
\rput(2,1){3}
\rput(1.5,1.5){3}

\end{pspicture}
\caption{}
\label{fig:vert_label}
\end{figure}
\end{center}

For a nonnesting partition $\nns$, a rank in the root poset, and a label, we define a function that records if there is a positive root in $\nns$ with that rank and label: \[\sll_{i,j}(\nns):=\begin{cases} j & \text{if } \exists \alpha \in \nns \text{ such that } \hgt(\alpha)=i, \sll(\alpha)=j \\ e & \text{otherwise} \end{cases},\] where $e$ is the identity element of $\W$.  Let $S^* \cup \{e\}$ be the set of words using simple reflections and $e$.

\begin{definition}
\label{def:bijA2}
Define $\bijVV: 2^{\Phi^+} \to S^* \cup \{e\}$ by
\[\bijVV(\nns) :=\prod_{i=n}^1 \left(\prod_{j=n}^1 s_{\sll_{i,j}(\nns)}\right)^{(-1)^i},\] where $s_e:=e.$ Define $\bijV: \NN(\W) \to \W$ by evaluating $\bijVV$ restricted to $\NN(\W)$ as an element of $\W$.
\end{definition}
\begin{lemma}\label{lem:stdcyclevert}
Let $i<j$, $i<i_1<i_2<\dots <i_k\leq j$. Then the product $$w=s_j s_{j-1}\cdots \hat{s}_{i_k}\cdots \hat{s}_{i_2}\cdots \hat{s}_{i_1}\cdots s_i s_{i+1}\cdots s_j$$
is (up to commutation) the standard form of an element $x\in\NC(\W, c)$ which is a cycle (here the hat denotes omission). In particular, if $p\in\NN(\W,c)$ satisfies $\mathrm{rk}(p)\leq 2$, then $\bijV(p)$ is (up to commutation) the standard form of a noncrossing partition which is a product of disjoint pairwise non nested cycles (see Example \ref{ex:cycle_vert}).
\end{lemma}
\begin{proof}
We argue by induction on $k$. If $k=0$ then $w$ is the standard form of the reflection $(i,j+1)$. Assume that the result holds for some $k-1\geq 0$. Then $$w=(s_{i_k-1} \cdots \hat{s}_{i_{k-1}}\cdots \hat{s}_{i_1}\cdots s_i s_{i+1}\cdots s_{i_k-1})(s_j s_{j-1}\cdots s_{i_k} s_{i_k+1}\cdots s_j).$$
By induction, the product $$(s_{i_k-1} \cdots \hat{s}_{i_{k-1}}\cdots \hat{s}_{i_1}\cdots s_i s_{i+1}\cdots s_{i_k-1})$$ is the standard form of a cycle $x'$ supported in $[i,i_k]$ (and with $x'(i_k)\neq i_k$) while $s_j s_{j-1}\cdots s_{i_k} s_{i_k+1}\cdots s_j$ is the standard form of $(i_k,j)$. But $x' (i_k,j)$ is a cycle which lies again in $\NC(\W,c)$ and its standard form is obtained by collapsing that of $x'$ with that of $(i_k,j)$.  
\end{proof}

\begin{example}\label{ex:cycle_vert}
Consider the order ideal of rank $2$:
\begin{center}
\begin{figure}[h!]
\begin{pspicture}(0,0)(6,1)
\pspolygon[linecolor=lightgray, linewidth=0.2pt, fillstyle=solid, fillcolor=lightgray](-1,-0.5)(0.5,1)(1,0.5)(1.5,1)(3,-0.5)(4,0.5)(4.5,0)(5.5,1)(7,-0.5)
\rput(0,0){{\bf 1}}
\rput(1,0){{\bf 2}}
\rput(2,0){{\bf 3}}
\rput(3,0){{ 4}}
\rput(4,0){{\bf 5}}
\rput(5,0){{\bf 6}}
\rput(6,0){{\bf 7}}
\rput(0.5,0.5){{\bf 2}}
\rput(1.5,0.5){{\bf 3}}
\rput(2.5,0.5){{ 4}}
\rput(3.5,0.5){{ 5}}
\rput(4.5,0.5){{ 6}}
\rput(5.5,0.5){{\bf 7}}
\end{pspicture}

\end{figure}
\end{center}
One has $$\bijVV(p)=\mathbf{s_7}\mathbf{s_3}\mathbf{s_2}\mathbf{s_1}\mathbf{s_2}\mathbf{s_3}\mathbf{s_5}\mathbf{s_6}\mathbf{s_7},$$
and hence $$\bijV(p)=s_7 s_3 s_2 s_1 s_2 s_3 s_5 s_6 s_7=(s_3s_2s_1s_2s_3)(s_5)(s_7s_6s_7),$$
which is exactly the standard form of $(1,4)(5,6,8)\in\NC(\W, c)$.
\end{example}
\begin{proposition}
\label{prop:nn_to_nc_vert}
The map $\bijV$ is a bijection from $\NN(\W)$ to $\NC(\W,c)$ such that $|\nns| = \ls(\bijV(\nns)).$ More precisely, the word $\bijVV(\nns)$ after evaluation yields an $\mathcal{S}$-reduced expression which is (up to commutation) the standard form of $\bijV(\nns)$.
\end{proposition}

\begin{proof}
For $i\in\{1,\dots,n\}$ and $p\in \NN(\mathfrak{S}_{n+1})$, we write $$M_p(i):=\max\{k~|~\exists \alpha\in p \text{ with }\sll(\alpha)=i\text{ and }\hgt(\alpha)=k \}.$$

Notice that for all $i$, $M_p(i)\leq\mathrm{rk}(p)$. Given an order ideal $p$, we cut horizontally at the heights $2k$ for $k=0,1,\dots,$ to obtain several bands and look at the connected components obtained at the various levels from our ordel ideal (see Figure \ref{fig:couper}). We claim that $\mathcal{V}(p)$ is a noncrossing partition, that the product gives a reduced expression for it, that each of the components obtained after cutting as above corresponds exactly to one cycle in the cycle decomposition of $\mathcal{V}(p)$ and that $j$ is not in the support of any cycle if and only if there exists $k=0,1,\dots$ such that $M_p(j)=2k$ and $M_p(j-1)\leq M_p(j)$. We show all these properties together by induction on the smallest integer $k$ such that $\mathrm{rk}(p)\leq 2k$. It is then clear that the standard form of each cycle is recovered when doing the product as above, hence that one recovers the standard form of a noncrossing partition, defined in Section~\ref{sec:nc_algebraic}. If $k=0$, then the only order ideal $p$ with $M_p(i)\leq 0$ for all $i$ is the empty order ideal, for which all the claimed properties obviously hold. 

Hence assume that the claimed properties hold for any order ideal $p$ with $2(k-2)\leq \mathrm{rk}(p)\leq 2(k-1)$ and assume that $2(k-1)\leq \mathrm{rk}(p)\leq 2k$. Consider the order ideal $p'$ obtained by cutting $p$ at height $2(k-1)$, that is
\[\bijVV(\nns') :=\prod_{i=2(k-1)}^{1} \left(\prod_{j=n}^1 s_{\sll_{i,j}(\nns')}\right)^{(-1)^i},\]
The word 
\[ \mathbf{w}:= \prod_{i=n}^{2(k-1)+1} \left(\prod_{j=n}^1 s_{\sll_{i,j}(\nns')}\right)^{(-1)^i}=\prod_{i=2k}^{2(k-1)+1} \left(\prod_{j=n}^1 s_{\sll_{i,j}(\nns')}\right)^{(-1)^i}\]
which must be added at the beginning of $\bijVV(\nns')$ to obtain $\bijVV(\nns)$ is the product of all the labels at height $2k$ from the right to the left, followed by the product of the labels at height $2k-1$ from the left to the right. Consider the last letter $\mathbf{s_i}$ in $\mathbf{w}$ and the largest $j$ such that $\mathbf{w'}:=\mathbf{s_j}\mathbf{s_{j+1}}\cdots\mathbf{s_{i}}$ is a terminal subword of $\mathbf{w}$. By maximality of $j$ there is no $\mathbf{s_{i-j-1}}$ at height $2k-1$ and hence at height $2k$, there is neither a $\mathbf{s_{i-j-1}}$ nor a $\mathbf{s_{i-j}}$.   This implies that the labels between $\mathbf{s_{i-j+1}}$ and $\mathbf{s_i}$ which may occur at height $2k$ commute in (the evaluation of) $\mathbf{w}$ with all the letters lying between $\mathbf{w'}$ and them, and can therefore be moved to the right just before $\mathbf{w}$. Now since we have the labels $\mathbf{s_j}\mathbf{s_{j+1}}\cdots\mathbf{s_{i}}$ at height $2k-1$, the labels $\mathbf{s_{i-j}},\cdots, \mathbf{s_{i+1}}$ must occur at height $2(k-1)$, hence $M_{p'}(\ell)=2(k-1)=\mathrm{rk}(p')$ for $\ell=i-j,\dots, i+1$. By induction, it follows that $i-j,\dots,i+1$ are not in the support of $p'$. Therefore $\bijV(p')$ fixes pointwise the interval $[i-j,i+1]$, hence multiplying $\bijV(p')$ by any noncrossing partition in the parabolic subgroup generated by the $s_{i-j},\dots, s_{i}$ still yields a noncrossing partition. By Lemma \ref{lem:stdcyclevert}, $\mathbf{w'}$ together with the product of the labels between $\mathbf{s_{i-j+1}}$ and $\mathbf{s_i}$ at height $2k$ (which have been moved just after $\mathbf{w'}$) is exactly the standard form (up to commutation and inverse) of a cycle with support in $[i-j,i+1]$ which fixes neither $i-j$ nor $i+1$, with an element $\ell\in[i-j,i+1]$ not in its support if and only if $\mathbf{s_{\ell}}$ occurs at height $2k$. This shows that multiplying $\bijVV(p')$ by the subword of $\mathbf{w}$ corresponding to the rightmost connected component of the labels at heights $2k-1$ and $2k$ still yields a noncrossing partition and shows all the claims; one argues exactly the same for the other connected components. 

We now show the property on the $M_p(\ell)$; if $\ell$ lies in  one of the sets $[i-j,i+1]$ corresponding to the added cycles as above, then we fall into an already-proved case.  We can therefore assume that $\ell$ is not in one of these sets---that is, that $M_{p'}(\ell)=M_p(\ell)$ and $M_{p'}(\ell-1)=M_{p}(\ell-1)$.  We first assume that $j$ is not in the support of any cycle of $p$. Since $p$ is obtained from $p'$ by adding blocks, then $j$ was not in the support of $p'$, so that by induction $M_{p'}(\ell-1)\leq M_{p'}(\ell)=2k'$. Conversely, assume that $M_p(\ell-1)\leq M_p(\ell)=2k'$. Then the same equality holds if we replace $p$ by $p'$, so that $j$ is not in the support of $p'$.  Since it is also not in the support of the added cycles, we conclude that it is not in the support of $p$.

\end{proof}

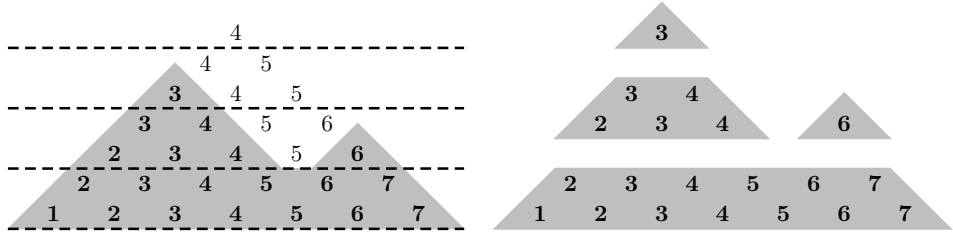
\begin{figure}
\begin{center}
\psscalebox{0.8}{
\begin{pspicture}(0,-0.5)(14,3.5)

\pspolygon[linecolor=lightgray, linewidth=0.2pt, fillstyle=solid, fillcolor=lightgray](-0.75,-0.25)(2,2.5)(3.75,0.75)(4.25,0.75)(5,1.5)(6.75,-0.25)
\psline[linestyle=dashed, linewidth=1.2pt](-0.75,0.75)(6.75,0.75)
\psline[linestyle=dashed, linewidth=1.2pt](-0.75,1.75)(6.75,1.75)
\psline[linestyle=dashed, linewidth=1.2pt](-0.75,-0.25)(6.75,-0.25)
\psline[linestyle=dashed, linewidth=1.2pt](-0.75,2.75)(6.75,2.75)
\rput(0,0){{\bf 1}}
\rput(1,0){{\bf 2}}
\rput(0.5,0.5){{\bf 2}}
\rput(2,0){{\bf 3}}
\rput(1.5,0.5){{\bf 3}}
\rput(1,1){{\bf 2}}
\rput(3,0){{\bf 4}}
\rput(2.5,0.5){{\bf 4}}
\rput(2,1){{\bf 3}}
\rput(1.5,1.5){{\bf 3}}
\rput(4,0){{\bf 5}}
\rput(3.5,0.5){{\bf 5}}
\rput(3,1){{\bf 4}}
\rput(2.5,1.5){{\bf 4}}
\rput(2,2){{\bf 3}}
\rput(5,0){{\bf 6}}
\rput(4.5,0.5){{\bf 6}}
\rput(4,1){5}
\rput(3.5,1.5){5}
\rput(3,2){4}
\rput(2.5,2.5){4}
\rput(6,0){{\bf 7}}
\rput(5.5,0.5){{\bf 7}}
\rput(5,1){{\bf 6}}
\rput(4.5,1.5){6}
\rput(4,2){5}
\rput(3.5,2.5){5}
\rput(3,3){4}

\pspolygon[linecolor=lightgray, fillstyle=solid, fillcolor=lightgray](7.25,-0.25)(8.25,0.75)(13.75,0.75)(14.75,-0.25)
\pspolygon[linecolor=lightgray, fillstyle=solid, fillcolor=lightgray](8.25,1.25)(9.25,2.25)(10.75,2.25)(11.75,1.25)

\pspolygon[linecolor=lightgray, fillstyle=solid, fillcolor=lightgray](9.25,2.75)(10,3.5)(10.75,2.75)

\pspolygon[linecolor=lightgray, fillstyle=solid, fillcolor=lightgray](12.25,1.25)(13,2)(13.75,1.25)

\rput(8,0){{\bf 1}}
\rput(9,0){{\bf 2}}
\rput(8.5,0.5){{\bf 2}}
\rput(10,0){{\bf 3}}
\rput(9.5,0.5){{\bf 3}}
\rput(9,1.5){{\bf 2}}
\rput(11,0){{\bf 4}}
\rput(10.5,0.5){{\bf 4}}
\rput(10,1.5){{\bf 3}}
\rput(9.5,2){{\bf 3}}
\rput(12,0){{\bf 5}}
\rput(11.5,0.5){{\bf 5}}
\rput(11,1.5){{\bf 4}}
\rput(10.5,2){{\bf 4}}
\rput(10,3){{\bf 3}}
\rput(13,0){{\bf 6}}
\rput(12.5,0.5){{\bf 6}}
\rput(14,0){{\bf 7}}
\rput(13.5,0.5){{\bf 7}}
\rput(13,1.5){{\bf 6}}

\end{pspicture}}
\end{center}
\caption{An example of how to read the cycles of $\bijV(p)$ from $p$. Each of the cycles in this example correspond to reflections; they are $(1,8)$, $(2,5)$, $(3,4)$, and $(6,7)$.}
\label{fig:couper}
\end{figure}

\begin{corollary}
        Let $\nns \in \NN(\W)$.  The rank of $\bijV(\nns)$ in absolute order is equal to 
        \[\lt(\bijV(\nns)):=|\{\alpha \in \nns : \hgt(\alpha)=1,3,5,\ldots\}| - |\{\alpha \in
\nns : \hgt(\alpha)=2,4,6,\ldots\}|.\]  In particular, there are
$\frac{1}{n+1}\binom{n+1}{k}\binom{n+1}{k-1}$ elements of $\NN(\W)$ with
$\lt(\bijV(\nns))=k$.
\end{corollary}

\begin{proof}
The proof of \ref{prop:nn_to_nc_vert} tells us how to read a cycle in the decomposition of $\mathcal{V}(p)$ from $p$. The length of the cycle corresponding to a component extracted from two consecutive rows is the difference in the lengths of those rows.
\end{proof}

\subsection{Diagonal Labeling}
\label{sec:nn_diagonal}

In this section, we provide a second labeling of the root poset, such that a product \emph{in a different order} using this new labeling gives a bijection to noncrossing partitions.  Although the order of the product is more complicated than in Section~\ref{sec:nn_vertical}, the labeling is correspondingly simpler.

Label the positive roots $\Phi^+$ by $\sl(e_i-e_j) = j-1$.  For example, in type $A_3$, $\Phi^+(\W)$ is labeled by 

\[\includegraphics[height=0.5in]{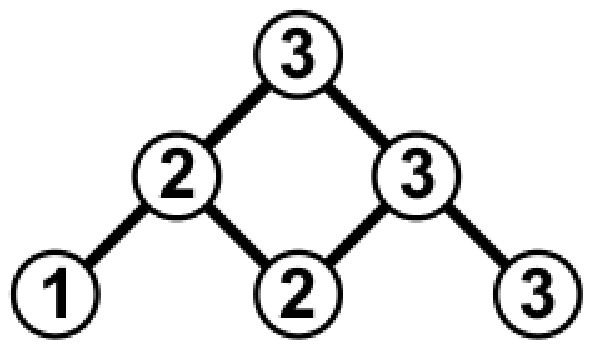}.\]

As before, for a nonnesting partition $\nns$, let \[\sl_{i,j}(\nns):=\begin{cases} j & \text{if } \exists \alpha \in \nns \text{ such that } \hgt(\alpha)=i, \sl(\alpha)=j \\ e & \text{otherwise.} \end{cases}\]

\begin{definition}
\label{def:bijA}
Define $\bijDD: 2^{\Phi^+} \to S^* \cup \{e\}$ by
\[\bijDD(\nns) := 
\left(\prod_{i \text{ even}} \prod_{j=n}^1 s_{\sl_{i,j}(\nns)}\right) \left(\prod_{i \text{ odd}} \prod_{j=n}^1 s_{\sl_{i,j}(\nns)}\right)^{-1},\] where $s_e:=e$ and both the product over odd rows and the product over even rows are in increasing order.
Define $\bijD: \NN(\W) \to \W$ by evaluating $\bijDD$ restricted to $\NN(\W)$ as an element of $\W$.
\end{definition}

\begin{example}
Let us consider the order ideal $\nns$ given in Figure \ref{fig:diag.ex}. Then $$\bijDD(\nns)=((\mathbf{s_5}\mathbf{s_4}\mathbf{s_3}\mathbf{s_2})(\mathbf{s_4}))((\mathbf{s_5}\mathbf{s_4}\mathbf{s_3}\mathbf{s_2}\mathbf{s_1})(\mathbf{s_4}\mathbf{s_3}))^{-1},$$
hence a straightforward computation yields $$\bijD(\nns)=(s_3s_2s_3)(s_5s_4s_3s_2s_1s_2s_3s_4s_5)=(2,4)(1,6)\in\NC(\mathfrak{S}_6,c).$$
\begin{center}
\begin{figure}[h!]
\psscalebox{0.8}{
\begin{pspicture}(0,0)(4,2)

\pspolygon[linecolor=lightgray, linewidth=0.2pt, fillstyle=solid, fillcolor=lightgray](-0.75,-0.25)(1.5,2)(3,0.5)(3.5,1)(4.75,-0.25)
\rput(0,0){{\bf 1}}
\rput(1,0){{\bf 2}}
\rput(0.5,0.5){{\bf 2}}
\rput(2,0){{\bf 3}}
\rput(1.5,0.5){{\bf 3}}
\rput(1,1){{\bf 3}}
\rput(3,0){{\bf 4}}
\rput(2.5,0.5){{\bf 4}}
\rput(2,1){{\bf 4}}
\rput(1.5,1.5){{\bf 4}}
\rput(4,0){{\bf 5}}
\rput(3.5,0.5){{\bf 5}}
\rput(3,1){{\bf 5}}
\rput(2.5,1.5){{\bf 5}}
\rput(2,2){{\bf 5}}
\end{pspicture}}
\label{fig:diag.ex}
\caption{}
\end{figure}
\end{center}

\end{example}
\begin{remark}
\label{rem:sort}
It is easy to see that this same labeling, with the map given instead by the product \[\mathbf{U}(I) := \prod_{i=1}^n \prod_{j=n}^1 s_{\sl_{i,j}(\nns)}\]
\noindent gives a bijection $\mathcal{U}$ between $\NN(\W)$ and $231$-avoiding permutations such that $|\nns| = \ls(\mathcal{S}(\nns)).$
\end{remark}

\begin{proposition}
\label{them:typea}
The map $\bijD$ is a bijection from $\NN(\W)$ to $\NC(\W,c)$ such that $|\nns| = \ls(\bijD(\nns)).$
\end{proposition}


For variety, we use the Kreweras complement defined on both noncrossing and nonnesting partitions to prove this, in a manner very similar in spirit to the ideas in \cite{ArmstrongStumpThomas} and \cite{StrikerWilliams}.  Note that this will also follow from Theorem~\ref{thm:all_equivalent}, which shows that $\bijD = \bijV$.

We show how to interpret the Kreweras complement on nonnesting partitions.  Given a nonnesting partition $\nns$, let $k$ be minimial so that $\alpha_{k+1} \not \in \nns$.  The \defn{initial part} $\Initial(\nns)$ is the nonnesting partition in the parabolic subgroup $\left\langle s_1,\dots, s_k\right\rangle$ that coincides with $\nns$ restricted to $\alpha_1,\alpha_2,\ldots,\alpha_k$.  The \defn{final part} $\Final(\nns)$ is the nonnesting partition in the parabolic subgroup $\left\langle s_{k+1},\dots, s_n\right\rangle$ that coincides with $\nns$ restricted to $\alpha_{k+1},\alpha_{k+2},\ldots,\alpha_{n}$.  

\begin{definition}
\label{def:krewmove}
Define the action $\KrewNN_c$ on $\nns \in \NN(\W)$ by replacing each root $e_i-e_j \in \Initial(\nns)$ by $e_i-e_{j-1}$ and each root $e_i-e_j \in \Final(\nns)$ by $e_{i-1}-e_j$.  The simple roots $\alpha_{k+1},\alpha_{k+2},\ldots,\alpha_n$ are added to $\Final(\nns)$.	
\end{definition}

\noindent It is not hard to check that $\KrewNN_c$ is invertible on $\NN(\W)$.

\begin{proof}[Proof of Proposition~\ref{them:typea}]
We consider two cases: if the nonnesting partition $\nns$ does not contain every simple root, or if $\nns$ does contain every simple root.  We will show that $|\nns| = \ls(\bijD(\nns))$ by proving that the word defined by $\bijD$ is a reduced $\mathcal{S}$-word.
\begin{enumerate}
	\item If the nonnesting partition $\nns$ does not contain every simple root, as parabolic subgroups of type $A$ are again of type $A$, we conclude the statement by induction on rank.
	
	\item Otherwise, we have a nonnesting partition $\nns$ containing every simple root.   It is immediate from Definition~\ref{def:krewmove} that $\bijD$ is equivariant with respect to $\KrewNN_c$ on such a nonnesting partition, since \[\bijD(\KrewNN_c(\nns)) = \bijD(\nns)^{-1} c=\Krew_c(\bijD(\nns)).\]  From Definition~\ref{def:krewmove}, $\KrewNN_c(\nns)$ removes the simple root $\alpha_n$ from the nonnesting partition $\nns$, so that $\KrewNN_c(\nns)$ falls into case (1).  Since $\Krew_c$ is invertible on $\Pc$, we again conclude by induction on rank that \[\bijD(\nns)=\Krew^{-1}_c(\Krew_c(\bijD(\nns)))=\Krew^{-1}_c(\bijD(\KrewNN_c(\nns)))=c \cdot \bijD(\KrewNN_c(\nns)))^{-1}\] is a bijection from nonnesting partitions that do contain every simple root to noncrossing partitions with reduced $S$-word containing all simple reflections.  By induction, $\bijD(\KrewNN_c(\nns))$ is reduced and contains no copy of the simple reflection $s_1$, so that $\bijD(\nns)=c \cdot \bijD(\KrewNN_c(\nns)))^{-1}$ is a parabolic decomposition with respect to $\{s_1\} \cup \{s_2,s_3,\ldots,s_{n}\}$ and is therefore reduced.  
\end{enumerate}

Since every nonnesting partition falls into one of the two cases, and since $|\NN(\W)|=|\Pc|=\frac{1}{n+2}\binom{2(n+1)}{n+1}$, we conclude the result.
\end{proof}

\begin{example}
\label{ex:typea}
In type $A_3$, the orbit of $\KrewNN_c$ on $\raisebox{-0.5\height}{\includegraphics[width=0.8in]{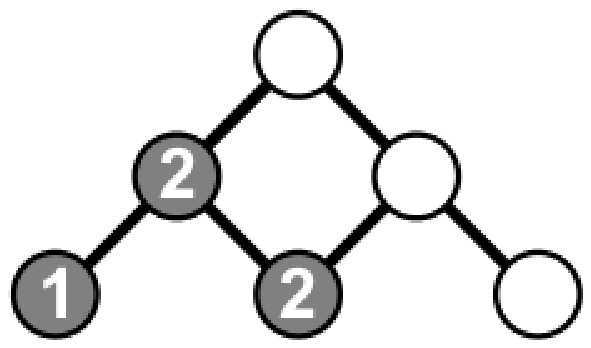}}$
 has period 4.  The first line of the table gives the orbit, while the second line records the factorizations of the corresponding noncrossing partitions under $\bijD$.
\begin{center}
$\begin{array}{|cccc|}
\hline 
 \includegraphics[width=1in]{a3oi1} & \includegraphics[width=1in]{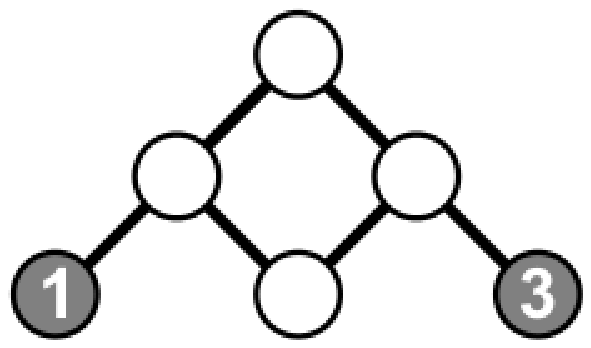} & \includegraphics[width=1in]{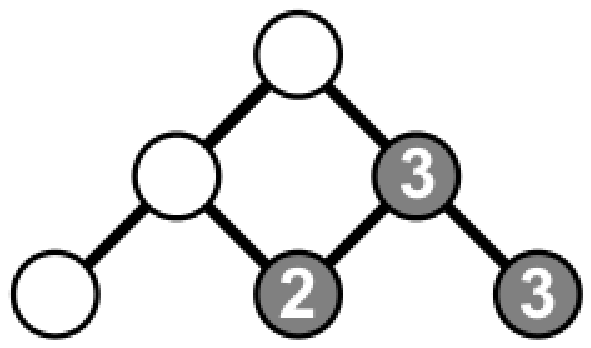} & \includegraphics[width=1in]{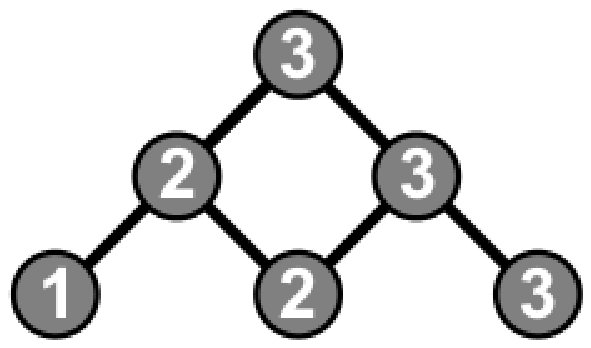}  \\ \hline
 s_2 s_1 | s_2 & s_3 s_1& s_3 s_2 | s_3& s_3 s_2 s_1 |s_3 |s_2 s_3 \\ \hline
\end{array}$
\end{center}
\end{example}

\subsection{Equivalences}

In this section, we show that all three bijections from Sections~\ref{sec:dyck},~\ref{sec:nn_vertical}, and~\ref{sec:nn_diagonal} are the same.

\begin{theorem}\label{thm:all_equivalent}
	For all $\nns \in \NN(\W),$ \[\Dyck(\nns)=\bijV(\nns)=\bijD(\nns).\]
\end{theorem}

\begin{proof}
	We first show $\bijV=\bijD$ by converting the product $\bijDD(\nns)^{-1}$ into the product $\bijVV(\nns)^{-1}.$  As $\nns$ is an order ideal, we first observe that conjugating the labels $\sll$ of an odd row $2k+1$ through each even row $2i$ for $1 \leq i \leq k$ that lies below it has the simple effect of increasing each label by the number of such even rows.  The highest row on which the root $e_i-e_j$ occurs is row $j-i$. When $j-i$ is odd, there are $\frac{j-i-1}{2}$ even rows below, so that $\sll(e_i-e_j)+\frac{j-i-1}{2}  = \lceil \frac{j+i-1}{2} \rceil+\frac{j-i-1}{2} = j-1 = \sl(e_i-e_j)$.  Similarly, conjugating the labels $\sll$ of an even row $2k$ through each even row $2i$ for $1 \leq i < k$ that lies below it also increases each label by the number of such even rows.  When $j-i$ is even, there are $\frac{j-i-2}{2}$ even rows below, so that $\sll(e_i-e_j)+\frac{j-i-2}{2}  = \lceil \frac{j+i-1}{2} \rceil+\frac{j-i-2}{2} = j-1 = \sl(e_i-e_j)$.  We compute
	
\begin{align*}	
\bijVV(\nns)^{-1} &= \prod_{i=1}^n \left(\prod_{j=1}^n s_{\sll_{i,j}(\nns)}\right)^{(-1)^i} \\
 &= \prod_{i \text{ odd}} \left(\prod_{j=1}^n s_{\sll_{i,j}(\nns)+(j-i-1)/2}\right)^{-1} \left( \prod_{i \text{ even}} \prod_{j=1}^n s_{\sll_{i,j}(\nns)}\right) \\
&=\left( \prod_{i \text{ odd}} \prod_{j=n}^1 s_{\sl_{i,j}(\nns)}\right)
\left( \prod_{i \text{ even}}\prod_{j=n}^1 s_{\sll_{i,j}(\nns)+(j-i-2)/2}\right)^{-1}\\
&=\left(\prod_{i \text{ odd}} \prod_{j=n}^1 s_{\sl_{i,j}(\nns)}\right)
\left(\prod_{i \text{ even}} \prod_{j=n}^1 s_{\sl_{i,j}(\nns)}\right)^{-1} = \bijDD(\nns)^{-1}.
\end{align*}	

We now show that $\Dyck = \bijV$. In the proof of Proposition~\ref{prop:nn_to_nc_vert}, we saw how to recover the various cycles of $\bijV(p)$ from $p$. Given a Dyck path, the beginning of a cycle in the corresponding noncrossing partition is given by two consecutive steps up, the first one beginning at a point with even coordinates, while the end of a cycle is given by two consecutive steps down, the last one ending at a point with even coordinates. Hence the cycles are extracted from the Dyck path in the same way as they are extracted from $p$, and it is then easy to see that the various obtained cycles must be equal.
\end{proof}

\section{Vectors}
\label{sec:vectors}

In this section, we define two sets of vectors corresponding to the labelings given in Sections~\ref{sec:nn_vertical} and~\ref{sec:nn_diagonal}.

\subsection{Vertical Vectors}\label{sub:vert}






\begin{definition}\label{defn:vertvect}
	For $\nns \in \NN(\W),$ $i\in[n],$ define \[\sll_{i}(\nns):=\left| \{ \alpha \in \nns : \sll(\alpha)=i\}\right|.\]
\end{definition}

It follows from the proof of Proposition~\ref{prop:nn_to_nc_vert} that $\sll_{i}(\nns)$ is equal to the number of copies of $s_i$ in the standard form of $\bijV(\nns)$.

For the noncrossing partition $x$ such that $\bijV(\nns)=x$, we can characterize $\sll_{i}(\nns)$ in a purely combinatorial way on the graphical representation of $x$. We write $P_1,\dots, P_r$ for the elements of $\Pol(x)$. For a polygon $P=[i_1 i_2 \cdots i_k]$, an integer $i\in \{2,\dots,n\}$ is \defn{nested} in $P$ if there exists $1<j\leq k$ with $i_{j-1}<i<i_j$.  Then we have  \[\sll_{i}(\nns)=2|\{P_m~|~i~\text{is nested in }P_m\}|+\mathbbm{1}_{i\in D_x},\]
where $\mathbbm{1}_{i\in D_x}=\begin{cases} 1 & \text{if } i\in D_x\\ 0 & \text{if } i\notin D_x\end{cases}.$  We write $\Nest(i):=|\{m~|~i~\text{is nested in }P_m\}|,$ omitting the dependance on $x$, and define $v_x:=(\sll_{i}(\nns))_{i=1}^n \in (\mathbb{Z}_{\geq 0})^n.$

For convenience we will write $x_i:=\sll_{i}(\nns).$

We now use both the vertical labeling of the root poset and this combinatorial way of reading the vertical vectors directly on the graphical representation of noncrossing partitions to characterize them.
\begin{lemma}\label{cor:injective}
Let $x,y\in \NC(\W, c)$. Then $x=y$ if and only if $v_x=v_y$. 
\end{lemma}

\begin{proof}
Thanks to the previous section, we have that $x=y$ if and only if $p_x=p_y$. 
\end{proof}

\begin{lemma}
\label{lem:parity}
For $x \in \NC(\W,c)$, the vector $v_x$ has the following properties:
\begin{enumerate}
	\item If $x_i$ is even and $x_{i+1}>x_i$, then $x_{i+1}=x_i+1$.
	\item If $x_i$ is odd and $x_{
    i+1}>x_i$, then $x_{i+1}=x_i+1$ or $x_i+2$.
	\item If $x_i$ is even and $x_{i+1}<x_i$, then $(x_{i+1}=x_i-1$ or $x_{i+1}=x_i-2)$ and $(i+1\in U_x)$.
	\item If $x_i$ is odd and $x_{i+1}<x_i$, then $x_{i+1}=x_i-1$ and $i+1\in U_x$ ($i+1$ is even maximal in its polygon).
	\item The integer $i\neq n+1$ lies in $U_x$ in exactly two situations:\begin{enumerate} 
	\item if $x_i$ is odd and $(x_{i-1}=x_i$ or $x_{i-1}=x_i+1)$ 
	\item if $x_i$ is even and $x_i<x_{i-1}$. 
\end{enumerate}
The integer $i=n+1$ lies in $U_x$ if and only if $x_{i-1}>0$.
\end{enumerate}
\end{lemma}
\begin{proof}
$1$ and $3$. If $x_i$ is even, then $\Nest(i)=\Nest(i+1)$ except in case $i+1\in U_x$ where $\Nest(i+1)=N(i)-1$. In the latter case, $x_{i+1}=x_i-2$ if $i+1\notin D_x$ and $x_{i+1}=x_i-1$ if $i+1\in D_x$. If $\Nest(i)=\Nest(i+1),$ $x_i=x_{i+1}$ if $i+1\notin D_x$ and $x_{i+1}=x_i+1$ if $i+1\in D_x$. 

$2$ and $4$. In this case, $i\in D_x$. Therefore, there exists $P\in\Pol(x)$ with $i\in P$ but $i$ is not terminal in $P$. If $i+1\in P$ and is terminal then $x_{i+1}=x_i-1$ since $i\in D_x$, $i+1\notin D_x$ and $\Nest(i)=\Nest(i+1)$. If $i+1\in P$ and is not terminal then $x_i=x_{i-1}$. If $i+1\notin P$ then $i+1$ is nested in $P$ (since $i\in D_x\cap P$) implying $\Nest(i+1)=\Nest(i)+1$. We then have $x_{i+1}=x_i+1$ if $i+1\notin D_x$ and $x_{i+1}=x_i+2$ if $i+1\in D_x$ (which is possible if $i+1$ is initial).

$5$. First assume that $i\in U_x$ and let $P\in\Pol(x)$ with $i\in P$. If $i\in D_x$, then $x_i$ is odd and $x_{i-1}=x_i$ if $i-1\in P$ and $x_{i-1}=x_i+1$ if $i-1\notin P$ (in which case $i-1$ has to be nested in $P$ since $i\in U_x$ but cannot lie in $D_x$). If $i\notin D_x$ then $x_i$ is even.  In this case, $x_{i-1}=x_i+1$ if $i-1\in P$ and $x_{i-1}=x_i+2$ if $i-1\notin P$ since in that case $i-1$ must be nested in $P$ (because $i$ is maximal in $P$). For the converse, all the cases where $x_i\neq x_{i-1}$ are given by $3$ and $4$. It remains to show that if $x_i=x_{i-1}$ and $x_i$ is odd, then $i\in U_x$. If $i$ and $i-1$ are not lying in the same polygon $P$, the assumption $i-1\in D_x$ implies that $i$ is nested in $P$ which contradicts $x_i=x_{i-1}$. Hence they lie in the same polygon $P$ and $i\in U_x$. Now consider the case where $i=n+1$: if $n+1\in U_x$ then there is $P\in\Pol(x)$ with $n+1\in P$. If $n\in P$, then $x_n=1$; if $n\notin P$, then $n$ is nested in $P$ and $x_n=2$. Conversely if $x_n>0$, then either $n\in D_x$ forcing $n+1\in U_x$ or $n$ is nested in a polygon $P$ which therefore needs to have $n+1$ as vertex. 
\end{proof}

\begin{proof}
If $x\neq y$, then $(D_x,U_x)\neq(D_y,U_y)$ by Proposition~\ref{prop:terminalinitial}. Since $i\in D_x$ if and only if $x_i$ is odd, by point $5$ of Lemma~\ref{lem:parity}, $v_x\neq v_y$.
\end{proof}

\begin{proposition}\label{prop:bij}
The map $x\mapsto v_x$, $x\in \NC(\W, c)$ defines a bijection from $\NC(\W,c)$ to the set of vectors $w=(w_i)_{i=1}^n\in(\mathbb{Z}_{\geq 0})^n$ with the following properties:
\begin{itemize}
	\item If $i$ is the smallest $i$ with $w_i\neq 0$, then $w_i=1$.
	\item If $i$ is the largest $i$ with $w_i\neq 0$, then $w_i=1$ or $w_i=2$.
	\item If $w_i$ is even and $w_{i+1}>w_i$ then $w_{i+1}=w_i+1$.
	\item If $w_i$ is odd and $w_{i+1}>w_i$ then $w_{i+1}=w_i+1$ or $w_{i+1}=w_i+2$.
	\item If $w_i$ is even and $w_{i+1}<w_i$, then $w_{i+1}=w_i-1$ or $w_{i+1}=w_i-2$.
	\item If $w_i$ is odd and $w_{i+1}<w_i$, then $w_{i+1}=w_i-1$.
\end{itemize}
\end{proposition}
\noindent It is convenient to represent the last four conditions as follows:
\begin{center}
\begin{tabular}{|l|c|r|}
  \hline
  $w_{i+1}-w_i$ & $w_{i+1}$ even & $w_{i+1}$ odd \\
  \hline
  $w_i$ even & $-2$ or $0$ & $1$ or $-1$ \\
  \hline
  $w_i$ odd & $1$ or $-1$ & $2$ or $0$ \\
  \hline
\end{tabular}
\end{center}
Examples of vectors satisfying these conditions are given in example \ref{exple:v}.
\begin{proof}
The injectivity of $x\mapsto v_x$ is given by \ref{cor:injective}. We show surjectivity using the vertical labeling of the root poset. We show by induction on the sum of the components of the vector $w$ that there exists an ordel ideal $p$ with vertical labeling giving the vector $w$. If $w$ is the zero vector then the empty ideal has the suitable labeling. Now assume that $w$ has nonzero entries and consider any index $i$ such that the corresponding entry $w_i$ is maximal. If $w_i$ is even then the conditions imply that $w_{i+1}=w_i, w_i-1$ or $w_i-2$ and $w_{i-1}=w_i$ or $w_i-1$. In particular, replacing $w_i$ by $w_i-1$ gives a vector $w'$ which still satisfies the conditions, so that by induction there exists an order ideal $p'$ with labeling giving $w'$. But since the largest height in $p'$ where there is a root with label $i$ is $w'_i$ which is odd and $w'_{i-1}=w_{i-1}\geq w'_i$, there is also a root at height $w'_i$ with label $i-1$ hence we can add the root $\alpha$ with label $i$ at height $w'_i+1$ to $p'$ giving again an order ideal $p:=p'\cup\{\alpha\}$ with labeling giving $w$. 

If $w_i$ is odd, then $w_{i-1}=w_i$, $w_i-1$ or $w_i-2$ and $w_{i-1}=w_i$ or $w_i-1$. In particular, replacing $w_i$ by $w_i-1$ yields a vector $w'$ which still satisfies the conditions above with corresponding order ideal $p'$. But since the biggest height in $p'$ where there is a root with label $i$ is $w'_i$ which is even and $w'_{i+1}=w_{i+1}\geq w'_i$, there is also a root with label $i+1$ at height $w'_i$ hence we can add the root $\alpha$ with label $i$ at height $w'_i+1$ to $p'$ giving again an order ideal $p:=p'\cup\{\alpha\}$ with labeling giving $w$.

\end{proof}
We will denote by $\Vn$ the set of vectors in $(\mathbb{Z}_{\geq 0})^n$ satisfying the properties of Proposition \ref{prop:bij}. By the bijection in Proposition~\ref{prop:bij}, there are Catalan many vectors in $\Vn$.
\begin{example}\label{exple:v}
The five elements of ${\sf{V}}_2$ are $(0,0)$, $(1,0)$, $(0,1)$, $(1,1)$, $(1,2)$. The fourteen elements of ${\sf{V}}_3$ are $(0,0,0)$, $(1,0,0)$, $(0,1,0)$, $(0,0,1)$, $(1,1,0)$, $(1,0,1)$, $(0,1,1)$, $(1,1,1)$, $(1,2,0)$, $(0,1,2)$, $(1,2,1)$, $(1,1,2)$, $(1,2,2)$, $(1,3,2)$.  
\end{example}

\subsection{Diagonal Vectors}



The definition corresponding to Definition~\ref{defn:vertvect} on the diagonal labeling of Section~\ref{sub:vert} results in a set of vectors that are trivially seen to be counted by the Catalan numbers.

\begin{definition}
	For $\nns \in \NN(\W),$ define \[\sl_{i}(\nns):=\left| \{ \alpha \in \nns : \sl(\alpha)=i\}\right|\] and let $u_\nns := (\sl_{i}(\nns))_{i=1}^n.$
\end{definition}

\begin{proposition}
	The map $p \mapsto u_p$ is a bijection from $\NN(\W,c)$ to the set of vectors $w=(w_i)_{i=1}^n$ such that $1-w_1 \leq 2-w_2 \leq \cdots \leq n-w_n$ and $w_i \leq i$.
\end{proposition}

\begin{proof}
	Taking the complement of $w$ inside the root poset gives a Ferrers shape inside a staircase.
\end{proof}

\section{Bruhat Order}
\label{sec:bruhat}

In this section, we prove that the poset of order ideals ordered by inclusion is isomorphic to the poset of noncrossing partitions ordered by the restriction of the Bruhat order by showing that the bijection $\bijD=\bijV=\Dyck$ is a poset isomorphism.  We denote by $\leq$ the Bruhat order on $\W$. For any two noncrossing partitions $x$ and $y$, their nonnesting partitions are related by $p_x\leq p_y$ if and only if $x_i\leq y_i$ for all $i\in [n]$, where $x_i$ and $y_i$ are the coordinates of the vertical vectors $v_x=(x_i)_{i=1}^n$ and $v_y=(y_i)_{i=1}^n$.   Thus, the following criterion is equivalent to establishing the poset isomorphism.

\begin{theorem}\label{thm:caract}
Let $x, y\in \NC(\W, c)$. Then $$x\leq y\Leftrightarrow \forall i, x_i\leq y_i.$$
\end{theorem}

\noindent An immediate corollary is:
\begin{corollary}\label{cor:lattice}
The poset $(\NC(\W,c), \leq)$ is a graded distributive lattice. 
\end{corollary}

The next two subsections are devoted to two different proofs of Theorem~\ref{thm:caract}, using the vertical and diagonal labelings, respectively.

\subsection{Proof of the lattice property using the vertical labeling}
We will use the tableau criterion given by Theorem 2.1.5 of \cite{B} which we recall below. We use the same notation as in \cite{B}---that is, for $x\in\mathfrak{S}_{n+1}$, $i, j\in\{1,\dots,n+1\}$, $$x[i,j]:=|\{a\in [i]~|~x(a)\geq j\}|.$$
\begin{theorem}[\cite{B}, Theorem 2.1.5]\label{thm:brenti}
Let $x,y\in\mathfrak{S}_{n+1}$. The following are equivalent:
\begin{enumerate}
\item $x\leq y$
\item $x[i,j]\leq y[i,j]$, for all $i,j\in\{1,\dots,n+1\}$.
\end{enumerate}
\end{theorem}
\noindent According to \cite{BjBr}, this criterion is due to Ehresmann \cite{Eh}. We now give a proof of Theorem \ref{thm:caract}.
\begin{proof}[Proof of \ref{thm:caract}]
Suppose that there exists $i$ with $x_i>y_i$. With the combinatorial characterization of the vector $v_x$ using the graphical representation of $x$ (see Section \ref{sub:vert}) we see that $x[i-1,i+1]=j$ if $x_i=2j$ or if $x_i=2j+1$, and $x[i,i+1]=j$ if $x_i=2j$ and $j+1$ if $x_i=2j+1$. If $x_i$ and $y_i$ have the same parity we then have $x[i,i+1]>y[i,i+1]$. If $x_i$ is even and $y_i$ is odd we have $x[i-1,i+1]>y[i-1,i+1]$. If $x_i$ is odd and $y_i$ is even we have $x[i,i+1]>y[i,i+1]$. Hence by Theorem \ref{thm:brenti} we have $x\not\leq y$.

Conversely, assume that $x_i\leq y_i$ for all $i$. We know that there exist order ideals $p_x$, $p_y$ such that the corresponding vertical labeling gives the vectors $v_x$ and $v_y$. Our assumption implies that $p_x\leq p_y$. Since $\mathcal{S}$-reduced words for $x$ and $y$ are then obtained by taking the product of the labels in a specific order, we conclude that the reduced word for $x$ is a subword of that for $y$, hence that $x\leq y$. 
\end{proof}

\subsection{Proof of the lattice property using the diagonal labeling}
\begin{lemma}
\label{lem:non_max}
Let $\nns \in \NN(\W)$.  If we delete a non-maximal root $e_i-e_j$ from $\nns$ to make the subset $\nns'$, then the word $\bijDD(\nns')$ (after ignoring all letters $e$) is either not reduced or $\bijD(\nns')$ is not a noncrossing partition.
\end{lemma}
\begin{proof}
We conclude the lemma by induction on rank if $\nns$ does not contain all simple roots.  If all simple roots are in $\nns$ and we delete a root $e_i-e_j$ that isn't simple, then by considering $\KrewNN_c(\nns)$ and instead deleting $e_{i}-e_{j-1}$, we again conclude the lemma by induction on rank.

The remaining case is therefore if all simple roots are in $\nns$ and $\nns'$ is obtained by deleting the simple root $e_i-e_{i+1}$.  Let $\bijD(\nns)$ have cycle decomposition \[\bijD(\nns)=(1,i_1,i_2,\ldots,i_k,n+1) \cdots,\] where $1$ and $(n+1)$ appear in the same cycle (by the assumption that all simple roots are in $\nns$) with $1 \leq i_1 \leq i_2 \leq \cdots, i_k \leq n+1$.  Note that deleting $e_i-e_{i+1}$ from $\nns$ is equivalent to multiplying $\bijD(\nns)$ on the left by the cycle $(1,i+1)=s_1s_2 \cdots s_i \cdots s_2s_1$.

Suppose that $i+1 \neq i_p$ for $1 \leq p \leq k$.  We compute the cycle decomposition \[\begin{aligned}\bijD(\nns')&=(1,i+1)\cdot (1,i_1,i_2,\ldots,i_k,n+1) \cdot (i+1,\ldots) \cdot x \\ &=(1,i_1,i_2,\ldots,i_k,n+1,i+1,\ldots)\cdot x,\end{aligned}\] where the cycle containing $i+1$ has been merged with the cycle containing $1$ and $n+1$.  This is manifestly not the cycle decomposition of a $c$-noncrossing partition since $1 < n+1 > i+1$.

Otherwise, $i+1=i_p$ for some $p$.  We compute the cycle decomposition \[\begin{aligned}\bijD(\nns')&=(1,i+1)\cdot (1,i_1,\ldots, i_{p-1},i+1,i_{p+1},\ldots,i_k,n+1) \cdot x \\ &=(1,i_1,\ldots,i_{p-1})(i+1,i_{p+1},n+1)\cdot x,\end{aligned}\] so that $\ls(\bijD(\nns'))=\ls(\bijD(\nns))-(2(i+1-i_{p-1})-1).$ This implies that the word $\bijDD(\nns')$, which has $\ls(\bijD(\nns))-1$ letters, is not reduced---unless $i=i_{p-1}$.  But if $i=i_{p-1},$ then it must have been that $e_{i-1}-e_{i}$ and $e_i-e_{i+1}$ were in $\nns$, but $e_{i-1}-e_{i+1}$ was not, in which case $e_i-e_{i+1}$ was maximal.
\end{proof}

The following proposition now gives an alternative proof of Theorem \ref{thm:caract}:
\begin{proposition}
The bijection $\bijD$ is an isomorphism from the distributive lattice $NN(\W)$ to Bruhat order on $\NC(\W, c)$.
\end{proposition}
\begin{proof}
Suppose that $I \lessdot J \in \NN(\W)$ is a cover relation.  Then by the definition of $\bijD$ and covers in $\NN(\W)$, we can write $\bijD(J)=b_1 b_2 \cdots b_i \cdots b_k$ with $b_j \in S$ and $\bijD(I)=b_1 b_2 \cdots \hat{b}_i \cdots b_k,$ where the circumflex denotes the deletion of the simple reflection.  By Theorem~\ref{them:typea}, both expressions are reduced, so we have 
\[\ls(\bijD(J))=\ls(\bijD(I))+1,\] and we can write 
\[\begin{aligned}
	\bijD(J)=b_1 b_2 \cdots b_i b_{i+1} \cdots b_k &= b_1 b_2 \cdots \hat{b}_i b_{i+1} \cdots b_k (b_k \cdots b_{i+1} b_i b_{i+1} \cdots b_k)\\&=\bijD(I)(b_k \cdots b_{i+1} b_i b_{i+1} \cdots b_k),
\end{aligned}\] so that $\bijD(I)\lessdot \bijD(J)$ in the Bruhat order.

Now suppose that $u \lessdot w$ with $ut = w$ is a cover relation in Bruhat order with $u,w \in \NC(\W,c)$.
From the bijection $\bijDD(\bijD^{-1}(w))$, we may factor $w$ into simple reflections corresponding to the roots in its nonnesting partition, $w = b_1 b_2 \cdots b_k$.  By the strong exchange property, $wt = b_1 \cdots \hat{b}_i \cdots b_k$ is a reduced word for $u$.  By assumption and using Lemma~\ref{lem:non_max}, the only possibilities are those $b_i$ corresponding to maximal elements of $\bijD^{-1}(w)$.  Therefore $\bijD^{-1}(u) \lessdot \bijD^{-1}(w)$ in $\NN(\W)$.
\end{proof}

\subsection{Covering relations}\label{sect:covering}

For $x,y\in\NC(\W, c)$, it follows from Theorem \ref{thm:caract} that $x$ is covered by $y$ in the Bruhat order if and only if $x<y$ and $\ell_{\mathcal{S}}(x)=\ell_{\mathcal{S}}(y)-1$: indeed, it is equivalent to say that the corresponding order ideals must satisfy $p_x<p_y$ with $p_y$ having one more element than $p_x$ and we now that the products of the labels in some order give reduced expressions for $x$ and $y$. In particular, the rank function is simply the Coxeter length. Theorem \ref{thm:caract} allows us to read off the covering relations on the nonnesting partitions side. The aim of this subsection is to describe the covering relations directly on $\NC(\W,c)$. To this end, we consider the combinatorial interpretation of $v_x$ using the graphical representation of $x$ as in Section \ref{sub:vert}. 

\begin{lemma}\label{lem:ancrage}
Let $x,y\in\NC(\W,c)$. Then $x$ is covered by $y$ if and only if $x$ is obtained from $y$ by one of the following operations:
\begin{itemize}
\item Replace a cycle of the form $(i_1,\dots,k, k+1,\dots, i_\ell)$ by the product of two cycles $(i_1,\dots,k)(k+1,\dots, i_\ell)$,
\item Replace a cycle of the form $(i_1,\dots, i_j, i_{j+1},\dots, i_\ell)$ where $i_j<k<i_{j+1}$ by the cycle $(i_1,\dots, i_j,k, i_{j+1},\dots, i_\ell)$.
\end{itemize}
\end{lemma}
\begin{proof}
The fact that both operations correspond to a covering relation in the Bruhat order is clear: in both cases, a copy of $s_k$ in the standard form $m_y$ of $y$ is deleted, and the resulting word yields the standard form $m_x$ of $x$ up to commutation. 

Now assume that $y$ covers $x$. We argue using the vertical vectors and their properties from Section \ref{sub:vert}. There is a unique $k$ such that $x_k\neq y_k$ and $y_k=x_k+1$. If $x_k$ is odd and $k\notin U_y$ consider the reflection $(i,j)$, $j>i$, with smallest $j-i$ and such that $(i,j)$ is an edge of a polygon $P$ of $y$ in which $k$ is nested (it always exists since $y_k$ is even and $y_k>0$). It suffices to add the vertex $k$ to the polygon $P$ to obtain an element $x'\in \NC(\W,c)$ satisfying $v_{x'}=v_x$ hence $x=x'$ by Corollary \ref{cor:injective}; this is possible since $k$ lies neither in $U_y$ nor in $D_y$. We replaced a cycle of the form $(\dots, i,j,\dots)$ in $y$ by the cycle $(\dots, i, k, j,\dots)$. If $x_k$ is odd and $k\in U_y$, we prove that $k-1\notin D_x$: if $k-1\in D_x$ then $k-1\in D_y$ hence $k-1$ and $k$ lie in the same polygon $P$ of $y$ with $k$ terminal (because $y_k$ is even) hence $$x_{k-1}=y_{k-1}=y_k+1=x_k+2$$ which by Lemma \ref{lem:parity} is impossible since both $x_{k-1}$ and $x_k$ are odd. But $k-1\notin D_x$ if and only if $k-1\notin D_y$, which implies that $y_{k-1}$ is even with $y_{k-1}>y_k$ since $k-1$ is nested in the polygon of $y$ having $k$ as vertex (because $k\in U_y$). But $y_k$ and $y_{k-1}$ are both even, hence we have a contradiction with Lemma \ref{lem:parity} again since $$x_{k-1}=y_{k-1}= y_k+2=x_k+3.$$ 

If $x_k$ is even, then $y_k$ is odd. In that case, there exists a polygon $P$ of $y$ such that $k$ is a non terminal vertex of $P$. Let $\ell$ be the vertex of $P$ following $k$. If $\ell\neq k+1$, one has a contradiction with Lemma \ref{lem:parity} since \[x_{k+1}=y_{k+1}\geq y_k+1=x_k+2.\] If $\ell=k+1$, splitting the polygon $P$ into two polygons by removing the edge $(k, k+1)$ yields an element $x'\in \NC(\W,c)$. This replaces a cycle $(\dots, k, k+1,\dots)$ from $y$ by the product of the cycles $(\dots, k)(k+1,\dots)$.
\end{proof}

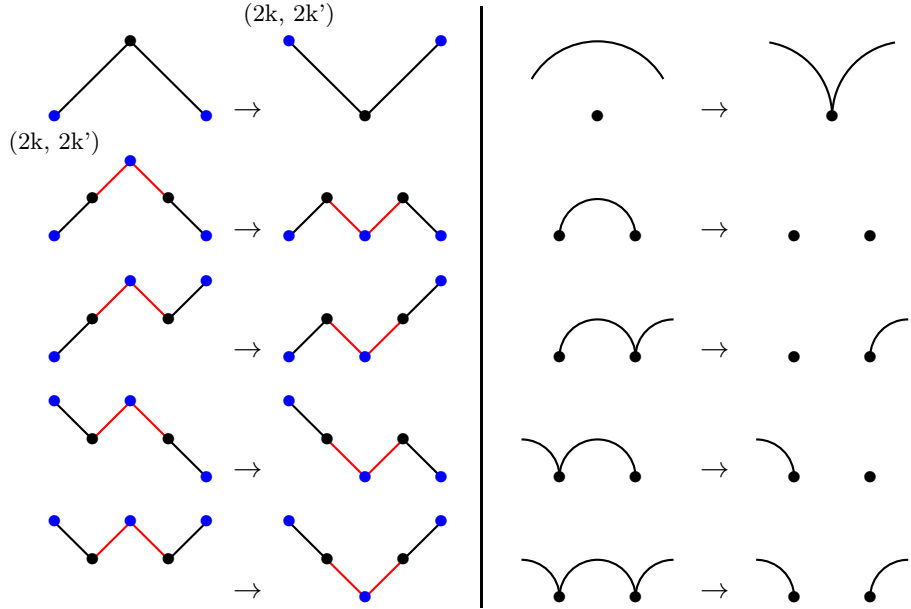
\begin{figure}[htbp]
\begin{center}
\begin{tabular}{cccc|cccc}
& & & & & & &\\
\begin{pspicture}(0,0)(2,1)
\psline(0,0)(1,1)
\psline(1,1)(2,0)
\rput(0,0){{\blue{$\bullet$}}}
\rput(1,1){$\bullet$}
\rput(2,0){{\blue{$\bullet$}}}
\rput(0,-0.35){\textrm{{\footnotesize (2k, 2k')}}}
\end{pspicture} & $\rightarrow$ & 
\begin{pspicture}(0,0)(2,1)
\psline(0,1)(1,0)
\psline(1,0)(2,1)
\rput(0,1){{\blue{$\bullet$}}}
\rput(1,0){$\bullet$}
\rput(2,1){{\blue{$\bullet$}}}
\rput(0,1.35){\textrm{{\footnotesize (2k, 2k')}}}
\end{pspicture} & & & \begin{pspicture}(0,0)(2,1)
\psarc(1,0){1}{30}{150}
\rput(1,0){$\bullet$}
\end{pspicture} & $\rightarrow$ & 
\begin{pspicture}(0,0)(2,1)
\psarc(0,0){1}{0}{80}
\psarc(2,0){1}{100}{180}
\rput(1,0){$\bullet$}
\end{pspicture}\\
& & & & & & &\\
\begin{pspicture}(0,0)(2,1)
\psline(0,0)(0.5,0.5)
\psline[linecolor=red](0.5,0.5)(1,1)
\psline[linecolor=red](1,1)(1.5,0.5)
\psline(1.5,0.5)(2,0)
\rput(0,0){{\blue{$\bullet$}}}
\rput(1,1){{\blue{$\bullet$}}}
\rput(0.5,0.5){$\bullet$}
\rput(1.5,0.5){$\bullet$}
\rput(2,0){{\blue{$\bullet$}}}
\end{pspicture} & $\rightarrow$ & \begin{pspicture}(0,0)(2,1)
\psline(0,0)(0.5,0.5)
\psline[linecolor=red](0.5,0.5)(1,0)
\psline[linecolor=red](1,0)(1.5,0.5)
\psline(1.5,0.5)(2,0)
\rput(0,0){{\blue{$\bullet$}}}
\rput(1.5,0.5){$\bullet$}
\rput(0.5,0.5){$\bullet$}
\rput(1,0){{\blue{$\bullet$}}}
\rput(2,0){{\blue{$\bullet$}}}
\end{pspicture}& & & \begin{pspicture}(0,0)(2,1)
\psarc(1,0){0.5}{0}{180}
\rput(0.5,0){$\bullet$}
\rput(1.5,0){$\bullet$}
\end{pspicture} & $\rightarrow$ & \begin{pspicture}(0,0)(2,1)
\rput(0.5,0){$\bullet$}
\rput(1.5,0){$\bullet$}
\end{pspicture}\\
& & & & & & &\\
\begin{pspicture}(0,0)(2,1)
\psline(0,0)(0.5,0.5)
\psline[linecolor=red](0.5,0.5)(1,1)
\psline[linecolor=red](1,1)(1.5,0.5)
\psline(1.5,0.5)(2,1)
\rput(0,0){{\blue{$\bullet$}}}
\rput(1,1){{\blue{$\bullet$}}}
\rput(0.5,0.5){$\bullet$}
\rput(1.5,0.5){$\bullet$}
\rput(2,1){{\blue{$\bullet$}}}
\end{pspicture} & $\rightarrow$ & \begin{pspicture}(0,0)(2,1)
\psline(0,0)(0.5,0.5)
\psline[linecolor=red](0.5,0.5)(1,0)
\psline[linecolor=red](1,0)(1.5,0.5)
\psline(1.5,0.5)(2,1)
\rput(0,0){{\blue{$\bullet$}}}
\rput(1.5,0.5){$\bullet$}
\rput(0.5,0.5){$\bullet$}
\rput(1,0){{\blue{$\bullet$}}}
\rput(2,1){{\blue{$\bullet$}}}
\end{pspicture}& & & \begin{pspicture}(0,0)(2,1)
\psarc(1,0){0.5}{0}{180}
\psarc(2,0){0.5}{90}{180}
\rput(0.5,0){$\bullet$}
\rput(1.5,0){$\bullet$}
\end{pspicture} & $\rightarrow$ & \begin{pspicture}(0,0)(2,1)
\rput(0.5,0){$\bullet$}
\rput(1.5,0){$\bullet$}
\psarc(2,0){0.5}{90}{180}
\end{pspicture}\\
& & & & & & &\\
\begin{pspicture}(0,0)(2,1)
\psline(0,1)(0.5,0.5)
\psline[linecolor=red](0.5,0.5)(1,1)
\psline[linecolor=red](1,1)(1.5,0.5)
\psline(1.5,0.5)(2,0)
\rput(0,1){{\blue{$\bullet$}}}
\rput(1,1){{\blue{$\bullet$}}}
\rput(0.5,0.5){$\bullet$}
\rput(1.5,0.5){$\bullet$}
\rput(2,0){{\blue{$\bullet$}}}
\end{pspicture} & $\rightarrow$ & \begin{pspicture}(0,0)(2,1)
\psline(0,1)(0.5,0.5)
\psline[linecolor=red](0.5,0.5)(1,0)
\psline[linecolor=red](1,0)(1.5,0.5)
\psline(1.5,0.5)(2,0)
\rput(0,1){{\blue{$\bullet$}}}
\rput(1.5,0.5){$\bullet$}
\rput(0.5,0.5){$\bullet$}
\rput(1,0){{\blue{$\bullet$}}}
\rput(2,0){{\blue{$\bullet$}}}
\end{pspicture}& & & \begin{pspicture}(0,0)(2,1)
\psarc(0,0){0.5}{0}{90}
\psarc(1,0){0.5}{0}{180}
\rput(0.5,0){$\bullet$}
\rput(1.5,0){$\bullet$}
\end{pspicture} & $\rightarrow$ & \begin{pspicture}(0,0)(2,1)
\rput(0.5,0){$\bullet$}
\rput(1.5,0){$\bullet$}
\psarc(0,0){0.5}{0}{90}
\end{pspicture}\\
& & & & & & &\\
\begin{pspicture}(0,0)(2,1)
\psline(0,1)(0.5,0.5)
\psline[linecolor=red](0.5,0.5)(1,1)
\psline[linecolor=red](1,1)(1.5,0.5)
\psline(1.5,0.5)(2,1)
\rput(0,1){{\blue{$\bullet$}}}
\rput(1,1){{\blue{$\bullet$}}}
\rput(0.5,0.5){$\bullet$}
\rput(1.5,0.5){$\bullet$}
\rput(2,1){{\blue{$\bullet$}}}
\end{pspicture} & $\rightarrow$ & \begin{pspicture}(0,0)(2,1)
\psline(0,1)(0.5,0.5)
\psline[linecolor=red](0.5,0.5)(1,0)
\psline[linecolor=red](1,0)(1.5,0.5)
\psline(1.5,0.5)(2,1)
\rput(0,1){{\blue{$\bullet$}}}
\rput(1.5,0.5){$\bullet$}
\rput(0.5,0.5){$\bullet$}
\rput(1,0){{\blue{$\bullet$}}}
\rput(2,1){{\blue{$\bullet$}}}
\end{pspicture}& & & \begin{pspicture}(0,0)(2,1)
\psarc(0,0){0.5}{0}{90}
\psarc(1,0){0.5}{0}{180}
\psarc(2,0){0.5}{90}{180}
\rput(0.5,0){$\bullet$}
\rput(1.5,0){$\bullet$}
\end{pspicture} & $\rightarrow$ & \begin{pspicture}(0,0)(2,1)
\rput(0.5,0){$\bullet$}
\rput(1.5,0){$\bullet$}
\psarc(0,0){0.5}{0}{90}
\psarc(2,0){0.5}{90}{180}
\end{pspicture}\\
\end{tabular}
\end{center}
\caption{Covering relations in the lattice of Dyck paths and the corresponding relations on noncrossing partitions. The parts in red are the parts which are changed. The points in blue are the points with even coordinates, that is, corresponding to the beginning or end of a step.}
\label{figure:couverture}
\end{figure}

\section{Extensions}\label{sec:extensions}

\subsection{Changing the Coxeter element}\markboth{CHANGING THE COXETER ELEMENT}{}

\subsubsection{Failure of the lattice property}\label{sub:fail}
The property that noncrossing partitions from a lattice under Bruhat order is specific to both type $A$ and to the linear Coxeter element $c$.  For example, $\NC(A_3, c')$ with  $c'=s_1 s_3 s_2$ does not form a lattice under Bruhat order: there are two maximal elements $x=s_3 s_2 s_1 s_2 s_3$ and $y=s_2 s_1 s_3 s_2$ in the poset (see Figure \ref{figure:a3}).  Similarly, $\NC(B_2,st)$ does not give a lattice (see Figure \ref{figure:b}). 

\begin{figure}[htbp]
\begin{pspicture}(-2,0)(10,6)
\psline(1.8,0.2)(1.2,0.8)
\psline(1.2,1.2)(1.8,1.8)
\psline(2.8,1.2)(2.2,1.8)
\psline(2.8,2.2)(2.2,2.8)
\psline(2.8,4.2)(2.2,4.8)
\psline(0.8,2.2)(0.2,2.8)
\psline(0.8,3.8)(0.2,3.2)
\psline(3.2,2.2)(3.8,2.8)
\psline(3.8,3.2)(3.2,3.8)
\psline(2,0.3)(2,0.7)
\psline(2,2.3)(2,2.7)
\psline(2,5.3)(2,5.7)
\psline(1,1.3)(1,1.7)
\psline(3,1.3)(3,1.7)
\psline(2.2,0.2)(2.8,0.8)
\psline(1.8,1.2)(1.2,1.8)
\psline(1.8,3.2)(1.2,3.8)
\psline(2.2,1.2)(2.8,1.8)
\psline(2.2,3.2)(2.8,3.8)
\psline(1.2,2.2)(1.8,2.8)
\psline(1.2,4.2)(1.8,4.8)
\rput(2,0){$e$}
\rput(1,1){$s_1$}
\rput(2,1){$s_2$}
\rput(3,1){$s_3$}
\rput(1,2){$s_1 s_2$}
\rput(2,2){$s_1 s_3$}
\rput(3,2){$s_2 s_3$}
\rput(2,3){$s_1 s_2 s_3$}
\rput(0,3){$s_2 s_1 s_2$}
\rput(4,3){$s_3s_2s_3$}
\rput(1,4){$s_2s_1s_2s_3$}
\rput(3,4){$s_1 s_3s_2s_3$}
\rput(2,5){$s_3s_2s_1s_2s_3$}
\rput(2,6){$s_2s_3s_2s_1s_2s_3$}
\psline(7.8,0.2)(7.2,0.8)
\psline(7.2,1.2)(7.8,1.8)
\psline(8.8,1.2)(8.2,1.8)
\psline(8.8,2.2)(8.2,2.8)
\psline(8.8,4.2)(8.2,4.8)
\psline(6.8,2.2)(6.2,2.8)
\psline(8.8,3.8)(8.2,3.2)
\psline(9.2,2.2)(9.8,2.8)
\psline(9.8,3.2)(9.2,3.8)
\psline(8,0.3)(8,0.7)
\psline(8,2.3)(8,2.7)
\psline(7,1.3)(7,1.7)
\psline(6.8,4.2)(7.65,5.8)
\psline(9.2,4.2)(8.35,5.8)
\psline(8.2,0.2)(8.8,0.8)
\psline(7.8,1.2)(7.2,1.8)
\psline(7.8,3.2)(7.2,3.8)
\psline(8.2,1.2)(8.8,1.8)
\psline(6.2,3.2)(6.8,3.8)
\psline(7.2,2.2)(7.8,2.8)
\psline(7.2,4.2)(7.8,4.8)
\psline(9,1.3)(9,1.7)
\rput(8,0){$e$}
\rput(7,1){$s_1$}
\rput(8,1){$s_2$}
\rput(9,1){$s_3$}
\rput(7,2){$s_2 s_1$}
\rput(8,2){$s_1 s_3$}
\rput(9,2){$s_2 s_3$}
\rput(8,3){$s_2 s_1 s_3$}
\rput(6,3){$s_2 s_1 s_2$}
\rput(10,3){$s_3s_2s_3$}
\rput(7,4){$s_2s_1s_2s_3$}
\rput(9,4){$s_3s_2s_3s_1$}
\rput(8,5){$s_2s_1s_3s_2$}
\rput(8,6){$s_3s_2s_1s_2s_3$}
\end{pspicture}
\caption{On the left is the Hasse digram of the restriction of Bruhat order to $\NC(\mathfrak{S}_4,s_1s_2s_3)$.  On the right is the restriction of Bruhat order to  $\NC(\mathfrak{S}_4,s_2s_1s_3)$.}
\label{figure:a3}
\end{figure}
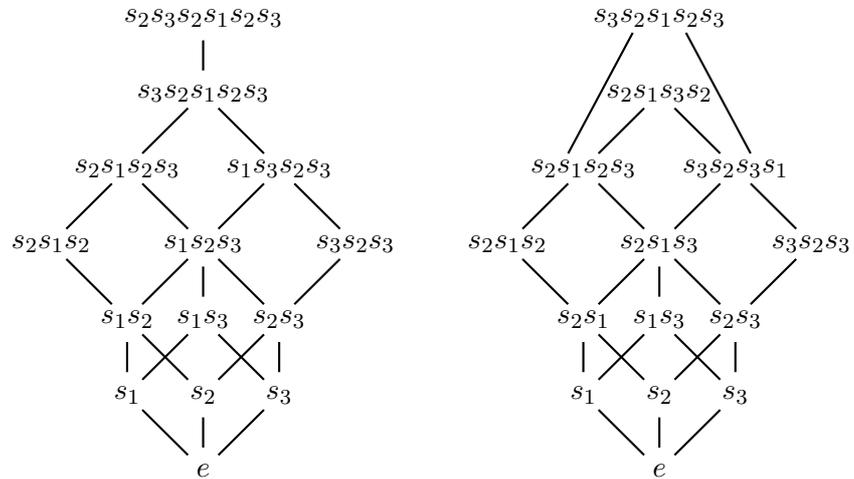
\begin{figure}[htbp]
$$\xymatrix{
sts \ar@{-}[rd] & & tst \ar@{-}[ld] &\\
 & st \ar@{-}[ld] \ar@{-}[rd] &\\
 s \ar@{-}[rd]& & t \ar@{-}[ld] \\
& e &}$$
\caption{The Hasse diagram of the restriction of Bruhat order to $\NC(B_2,st)$.}
\label{figure:b} 
\end{figure}
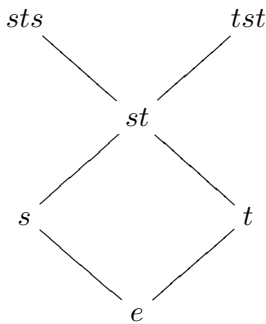

There are representation-theoretic reasons for believing that there ought to be an order on $\NC(\W, c')$ isomorphic to $(\NC(\W, c),\leq)$, as we now explain.  One natural base of the Temperley-Lieb algebra is the diagram basis; another is the basis given by the images of the simple elements of the dual braid monoid for linear $c$ (which are naturally indexed by elements of $\NC(\W,c)$).  Ordering the latter by Bruhat order conveniently induces an upper-triangular change-of-basis matrix to the former. 

But a basis arising from the dual braid monoid can be defined for an \emph{arbitrary} Coxeter element $c'$---as the image of the simple elements, this basis is now indexed by elements of $\NC(\W, c')$.  It turns out that to pass to the diagram basis, one can still find an order on $\NC(\W,c')$ that produces an upper-triangular change-of-basis matrix (see \cite{GobTh}).  Surprisingly, the partial order on $\NC(\W,c')$ required to obtain triangularity is isomorphic to the Bruhat order on $\NC(\W,c)$.  The goal of this section is to combinatorially define this order on $\NC(\W, c')$ by finding a bijection with $(\NC(\W,c),\leq)$.  We emphasize that our bijection will fix the set of reflections $\T$---in particular, it is a different bijection from the standard one given by conjugation.



Since the Bruhat order on $\NC(\W,c')$ no longer gives the correct order required for the triangularity above, we substitute the set $\Vn$ of vertical vectors under inclusion as a replacement.  More precisely, for an arbitrary Coxeter element $c'$, we define standard forms for elements of $\NC(\W,c')$, and recover the set $\Vn$ by counting copies of $s_i$ in these standard forms---generalizing the definition of $\Vn$ in Section~\ref{sec:nc_algebraic} for $c=s_1s_2\cdots s_n$.

The key is that we no longer insist that standard forms be \emph{reduced} words.

\begin{example}
\label{ex:mot1}
As a motivating example, consider type $\mathfrak{S}_4$ with $c'=s_2 s_1 s_3$ and $x'=s_2 s_1 s_3 s_2$.  Define the standard form of $x'$ to be the \emph{unreduced} word $(s_2 s_1 s_2) (s_3 s_2 s_3)$, so that $v_{x'} = (1,3,2)$.  Figure~\ref{figure:a3} gives all other standard forms for elements of $\NC(\W,c')$ (these are largely forced).  Then the ordering on $\NC(\W, c')$ induced by the vectors $\Vn$---obtained by counting copies of $s_i$---under componentwise order is isomorphic to $(\NC(\W,s_1s_2s_3), \leq)$.  
\end{example}

We will define general standard forms in the same way as we did for linear $c$: we first specify a canonical $\T$-reduced expression, and then replace each reflection by a specific $\mathcal{S}$-reduced expression.

\subsubsection{Geometric representation of noncrossing partitions for arbitrary Coxeter elements}\label{sub:orientation}

Let $c'$ be an arbitrary Coxeter element. As before, there is a graphical representation of elements of $\NC(\W, c')$ as disjoint unions of polygons whose vertices are given by labeled points around a circle. If we write the Coxeter element as $c'=(i_1, i_2,\dots,i_{n+1})$ with $i_1=1$, then the labels $1=i_1, i_2,\dots, i_{n+1}$ occur around the circle in clockwise order.  An example is given in Figure \ref{figure:orientation}.   Recall that not all $(n+1)$-cycles correspond to Coxeter elements---Lemmas~\ref{lemme1} characterizes which cycles can occur.


\begin{lemma}\label{lemme1}
An $(n+1)$-cycle $(i_1,\dots,i_{n+1})$ with $i_1=1$ and $i_k=n+1$ corresponds to a standard Coxeter element $c'$ iff $1=i_1<i_2<\dots <i_k$ and $1=i_1<i_{n+1}<i_{n}<\dots<i_{k+1}<n+1$.
\end{lemma}
\begin{proof}
Let $c'$ be a standard Coxeter element, and write $\sigma_{c'}$ for the permutation corresponding to $c'$. Since $\sigma_{c'}(1)=i_2$, $s_j$ must occur before $s_{j-1}$ in any product of the $s_i$ giving $c'$ for all $j\in \{2, 3, \dots, i_2-1\}$. Moreover, if $k\neq 2$, $s_{i_2}$ cannot be before $s_{i_2-1}$ because it would contradict $\sigma_{c'}(1)=i_2$. Hence $k\neq 2$ implies that $s_{i_2}$ is after $s_{i_2-1}$; but these two reflections are the only elements of $S$ that don't fix $i_2$. This implies that $i_3=\sigma_{c'}(i_2)>i_2$. Iterating this process gives $1=i_1<i_2<\dots <i_k$. A similar argument gives the second sequence of inequalities. 

For the converse, we argue by induction on $n$. If $n=2$, then such a cycle is either equal to $(1,2,3)$ or $(1,3,2)$. The first one is $s_1 s_2$ and the second one $s_2 s_1$. Now let $c'=(i_1,\dots,i_{n+1})$. If the inequalities of the lemma are true, then either $i_2=2$ or $i_{n+1}=2$. If $i_2=2$ then $s_1 c=(i_2,\dots,i_{n+1})$ and the result follows by applying the induction hypothesis to the $n$-cycle $(i_2,\dots,i_{n+1})$ in the parabolic subgroup $W_I$ where $I=\{2, \dots, n+1\}$. If $i_{n+1}=2$ then $c s_1=(i_2,\dots,i_{n+1})=(i_{n+1},i_2,\dots,i_{n})$ and we can apply the induction hypothesis to the $n$-cycle $(i_{n+1},i_2,\dots, i_{n})$ in the parabolic subgroup $W_I$.

\end{proof}

\begin{remark}\label{rmq:successifs}
It follows from Lemma~\ref{lemme1} that for any $k\in\{1,\dots, n+1\}$, there is a line $L_k$ passing through the point $k$ such that the set of points on the circle with label in $E_k:=\{i\in[n+1]~|~i\leq k\}$ lies in one of the two half-planes defined by $L_k$ while the set of points with label in ${E'}_k:=\{i\in[n+1]~|~i\geq k\}$ lies in the other half-plane. An example is given in Figure \ref{figure:orientation}.
\end{remark}
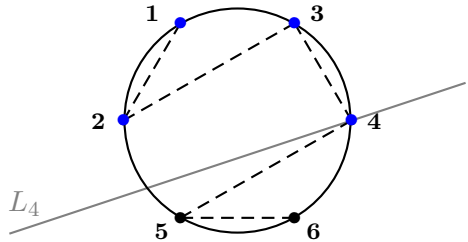
\begin{figure}[htbp]
\begin{center}
\begin{tabular}{ccc}
& \begin{pspicture}(0,0)(6,3)

\psline[linecolor=gray](0,0)(6,2)
\pscircle(3,1.5){1.5}
\psdots(4.5,1.5)(3.75,2.79)(3.75,.21)(2.25,.21)(2.25,2.79)(1.5,1.5)
\psline[linestyle=dashed](2.25,2.79)(1.5,1.5)
\psline[linestyle=dashed](1.5,1.5)(3.75,2.79)
\psline[linestyle=dashed](3.75,2.79)(4.5,1.5)
\psline[linestyle=dashed](4.5,1.5)(2.25,.21)
\psline[linestyle=dashed](2.25,.21)(3.75,.21)
\rput(3.75,.21){$\bullet$}
\rput(4.5,1.5){{\blue $\bullet$}}
\rput(3.75,2.79){{\blue $\bullet$}}
\rput(2.25,.21){$\bullet$}
\rput(2.25,2.79){{\blue $\bullet$}}
\rput(1.5,1.5){{\blue $\bullet$}}
\rput(4.05,2.92){\textrm{{\footnotesize \textbf{3}}}}
\rput(4.8,1.5){\textrm{{\footnotesize \textbf{4}}}}
\rput(4,.03){\textrm{{\footnotesize \textbf{6}}}}
\rput(2,.03){\textrm{{\footnotesize \textbf{5}}}}
\rput(1.17,1.5){\textrm{{\footnotesize \textbf{2}}}}
\rput(1.88,2.92){\textrm{{\footnotesize \textbf{1}}}}
\rput(0.2,.4){{\gray $L_4$}}
\end{pspicture} & \\
\end{tabular}
\end{center}
\caption{Example of a labeling given by the Coxeter element $c'=s_2 s_1 s_3 s_5 s_4=(1,3,4,6,5,2)$. When going from the point $1$ to the point $6$ one obtains a noncrossing zigzag.}
\label{figure:orientation}
\end{figure}

\begin{notation}
Let $c'$ be a Coxeter element. We set $R_{c'}:=\{i_1, i_2,\dots, i_k\}$ and $L_{c'}:=\{i_k, i_{k+1}, \dots, i_{n}, i_{n+1}, i_1\}$, where the $i_j$'s are given by Lemma \ref{lemme1}. In particular, $L_{c'}\cup R_{c'}=[n+1]$ and $L_{c'}\cap R_{c'}=\{1,n+1\}$.
\end{notation}

\subsubsection{Standard forms for cycles.}


Our goal is to define standard forms for noncrossing partitions associated to an arbitrary Coxeter element $c'$.  In this section, we first define the standard form for an individual cycle with respect to $c'$.

Let $x\in\NC(\W,c)$ and recall that we use the shorthand $x_i$ for the $i$\ts{th} entry of $v_x$, which counts the number of copies of $s_i$ in the standard form $m_{x}$.
\begin{lemma}[Standard forms for cycles]\label{lem:singleblock}
Let $x'\in \NC(\W, c')$ be a cycle and let $x$ be the unique cycle in $\NC(\W, c)$ with the same support as $x'$. There exists an $\mathcal{S}$-reduced expression $m_{x'}^{c'}$ for $x'$ such that:
\begin{itemize}
\item For each $1\leq i\leq n$, the number of copies of $s_i$ in $m_{x'}^{c'}$ is the same as in $m_x$,
\item If $\mathrm{supp}(x')=\{d_1,d_2,\dots, d_k\}$ where $d_i<d_{i+1}$ for $1\leq i<k$, then $m_{x'}^{c'}$ is a product of the syllables $\mathbf{s_{[d_i,d_{i+1}]}}$ in some order.
\end{itemize} 
\end{lemma}
\begin{proof}
We argue by induction on $\ell_{\mathcal{T}}(x')=|\mathrm{supp}(x')|-1$.  If $\ell_{\mathcal{T}}(x')=1$, then $x'\in \mathcal{T}\subset \NC(\W,c)\cap\NC(\W, c'),$ so that $x'=x$ and we can set $m_{x'}^{c'}=m_x$; we know that the standard form of an element in $\NC(\W, c)$ has the specified properties. 

We may now assume that $x'=(i_1, i_2,\dots, i_k)$ with $k>2$. Using the labeling of the circle arising from the cycle corresponding to $c'$ (as in Subsection \ref{sub:orientation}), we can assume that $i_1$ and $i_2$ are the two smallest indices in $\{i_1,i_2,\dots,i_k\}$. We can rewrite $x'$ as the product $(i_1, i_2)(i_2, i_3,\dots, i_k)$ in case $i_2>i_1$ or $(i_1, i_3,\dots, i_k)(i_1, i_2)$ in case $i_2<i_1$. Then $y'=(i_2, i_3,\dots, i_k)$ (or $y'=(i_1,i_3,\dots, i_k)$) lies again in $\NC(\W, c')$. By induction, there is an $\mathcal{S}$-reduced expression $m_{y'}^{c'}$ of $y'$ in which the number of copies of $s_i$'s is the same as in $m_y$ for $y$ the unique cycle in $\NC(\W,c)$ such that $\mathrm{supp}(y')=\mathrm{supp}(y)$. But the elements of $\mathcal{S}$ occuring in any reduced expression of $(i_1, i_2)$ are distinct from those occuring in $m_{y'}^{c'}$ and have smaller indices. As a consequence, the Coxeter word $\mathbf{s_{[i_1,i_2]}}\star m_y$ if $i_1<i_2$ or $\mathbf{s_{[i_2,i_1]}}\star m_y$ if $i_2<i_1$ is a standard form for the unique cycle $x\in \NC(\W,c)$ with $\mathrm{supp}(x)=\mathrm{supp}(x')$. Then $m_{x' }^{c'}=\mathbf{s_{[i_1,i_2]}}\star m_{y'}^{c'}$ if $i_1<i_2$, or $m_{x'}^{c'}=m_{y'}^{c'} \star\mathbf{s_{[i_2, i_1]}}$ otherwise, is a word with the desired properties.
\end{proof}
\begin{remark}\label{rmq:stdunique} For $c'=c$ we recover the standard form of a cycle from Section \ref{sec:nc_algebraic}, i.e.\, $m_x^c=m_x$ for any $x\in\NC(\W,c)$. For arbitrary $c'$, the word $m_{x'}^{c'}$ is not unique in general, since the induction of the proof allows for adjacent syllables $\mathbf{s_{[d_i, d_{i+1}]}}\star\mathbf{s_{[d_j, d_{j+1}]}}$ with $j>i+1$ (or $j+1<i$) to be placed at the end of $m_{x'}^{c'}$---and these syllables commute as elements of $\W$.  Regardless, the \textit{relative} positioning of the noncommuting syllables $\mathbf{s_{[d_i, d_{i+1}]}}$ and $\mathbf{s_{[d_{i+1},d_{i+2}]}}$ is always uniquely determined in $m_{x'}^{c'}$.
\end{remark}
\begin{definition}
A word $m_{x'}^{c'}$, as in Lemma \ref{lem:singleblock}, is a~\defn{standard form} of $x'$.
\end{definition}
\begin{example}
In type $A_4$ let $c'=s_4 s_2 s_1 s_3=(1,3,5,4,2)$ and $x'=(1,3,5,2)\in\NC(\W,c')$.  The corresponding element of $\NC(\W,c)$ is $x=(1,2,3,5)$, and $m_x=s_1 s_2 (s_4 s_3 s_4)$.  By Lemma \ref{lem:singleblock}, we compute that $m_{x'}^{c'}=s_2 (s_4 s_3 s_4) s_1$, which is just a reordering of the syllables of $m_x$.  The vector in $\sf{V}_4$ corresponding to $x'$ is therefore $(1,1,1,2)$.
\end{example}

  
\begin{definition}
If $x\in\NC(\W, c')$ and $\{d_1<\dots<d_k\}$ is the set of integers indexing the vertices of $P\in\Pol(x)$, we say that $d_1$ is an \defn{initial} index and $d_k$ a \defn{terminal} one. As for $c$, we set $D_x^{c'}$ (resp. $U_x^{c'}$) to be the set of labels of non-terminal (resp. non-initial) vertices of polygons of $x'$. 
\end{definition}

\subsubsection{Standard forms for noncrossing partitions.}

In this section, we will define a standard form $m_{x'}^{c'}$ for any $x'\in\NC(\W, c')$ by establishing a bijection $\phi_{c',c}:\NC(\W, c')\rightarrow\NC(\W, c)$ with the property that $m_{x'}^{c'}$ has the same number of $s_i$'s as $m_{\phi_{c',c}}$.  This bijection will extend the bijection given in Lemma \ref{lem:singleblock} for individual cycles.  In particular, since $\phi_{c',c}$ will fix the reflections $\T$, we note that $\phi_{c',c}$ is not the standard bijection from $\NC(\W,c')$ to $\NC(\W,c)$ obtained by conjugating.

The most naive definition of such a bijection $\phi_{c',c}$ would be to decompose $x'$ into a product of disjoint cycles $c_1'c_2'\cdots c_k'$, and then to concatenate the standard forms given by Lemma \ref{lem:singleblock} for each cycle $c_i'$.  The difficulty is that the element $x=\phi_{c',c}(x')\in\NC(\W,c)$ cannot in general be given by the product $c_1c_2\cdots c_k$ where $c_i$ is the element of Lemma \ref{lem:singleblock} corresponding to $c_i'$, since such a bijection would fix every element that is a product of commuting transpositions---and such elements do not necessarily lie in $\NC(\W, c)$.  The following example illustrates the failure of this naive approach.
\begin{example}
\label{ex:mot2}
To continue Example~\ref{ex:mot1}, observe that $\phi_{c',c}$ cannot fix the element $x'=s_2s_1s_3s_2 = (s_2s_1s_2)(s_3s_2s_3)\in\NC(\W, s_2s_1s_3)$, since $x' \not \in \NC(\W,c)$. Since $x'$ is the product of two commuting reflections, this shows that we cannot simply concatenate the standard forms given by Lemma \ref{lem:singleblock}.  From Example~\ref{ex:mot1} and Figure~\ref{figure:a3}, we would like $\phi_{c',c}(x')$ to be the element $s_2 s_3s_2s_1s_2s_3$.  Then we note that both $x$ and $x'$ have the same support, and that $(D_x, U_x)=(D_{x'}^{c'}, U_{x'}^{c'})$.
\end{example}

The observation in Example~\ref{ex:mot2} provides the solution in the form of the following generalization of Lemma~\ref{lem:singleblock}.

\begin{theorem}\label{thm:bijccprime}
There exists a unique bijection \[\phi_{c',c}:\NC(\W,c')\rightarrow\NC(\W,c), ~x'\mapsto x\] such that $(D_x, U_x)=(D_{x'}^{c'}, U_{x'}^{c'})$. Furthermore, $\lt(x')=\lt(x)$, and if $c_1c_2\dots c_j$ is the cycle decomposition of $x'$, then the word for $x'$
\[m_{x'}^{c'}:=m_{c_1}^{c'}\star\cdots m_{c_j}^{c'}\] has the same number of $s_i$'s as $m_x$.
\end{theorem}

A word $m_{x'}^{c'}$ as in the Theorem is called a \defn{standard form} of $x'$. We break the proof of Theorem~\ref{thm:bijccprime} into three parts: Propositions~\ref{lem:part1bijcc},~\ref{lem:ensb}, and~\ref{prop:nouvebij}.  In Proposition~\ref{lem:part1bijcc}, we define a map $\phi_{c',c}$ that takes an element of $\NC(\W,c')$ as input and prove that it produces an element of $\NC(\W,c)$ as output.  The injectivity of $\phi_{c',c}$ is then shown in Proposition~\ref{lem:ensb}.  Finally, in Proposition~\ref{prop:nouvebij}, we conclude that the standard form of $x'$ and the standard form of $\phi_{c',c}(x')$ have the same number of copies of $s_i$. Notice that if a bijection as in the Theorem exists, it is automatically unique, since the sets $(D_x, U_x)$ characterize the noncrossing partition by Proposition \ref{prop:terminalinitial}.

\begin{proposition}
\label{lem:part1bijcc}
There is a well-defined map \[\phi_{c',c}:\NC(\W,c')\rightarrow\NC(\W,c), ~x'\mapsto x.\]
\end{proposition}

\begin{proof}
We define the map $x'\mapsto x$ by representing $x'$ as a collection of arcs with crossings, and then resolve the crossings.

We fix a line with marked points from $1$ to $n+1$ (even for $c'\neq c$).  For each $P\in\Pol(x')$, we order the set $\{d_1,\dots, d_k\}$ of labels of $P$ so that $d_i<d_{i+1}$ for $i=1,\dots, k-1$, and represent $P$ by successive arcs joining the point on the line with label $d_i$ to the point with label $d_{i+1}$, for $i=1,\dots, k-1$.  Finally, we represent $x'$ as the collection of its polygons.  By Lemma \ref{lem:singleblock}, the element of $\NC(\W, c')$ corresponding to $P$ is given by a product of the reflections $(d_i, d_{i+1})$ in some order, depending on $c'$.

Since the points on the line are labeled from $1$ to $n+1$, for $c'\neq c$ the resulting diagram may have crossings.  We will now supply a algorithm to eliminate these crossings, and prove that it gives a bijection to $\NC(\W,c)$.

If the diagram associated to $x'$ has no crossings, then it is the diagram of an element $x\in\NC(\W,c)$, and the procedure terminates.  If there are at least two arcs $(i, k)$, $(j,\ell)$ which cross each other, say with $i<j<k<\ell$, we replace them by the two noncrossing arcs $(i,\ell)$, $(k, j)$ and repeat.

At each step, the number of crossings decreases by one, so that the algorithm terminates in a diagram with no crossings---which represents an element $x\in\NC(\W,c)$.  Moreover, since the operations only change local configurations of the diagram in small neighborhoods of the crossings, the order in which we resolve the crossings does not affect the resulting diagram. The algorithm given above is therefore a well-defined map $\phi_{c',c}:\NC(\W,c')\rightarrow \NC(\W,c)$, $x'\mapsto x$; Figure \ref{figure:process} illustrates an example. It is clear that the set of vertices; the sets of initial, terminal, non initial, and non terminal vertices; and the absolute length are all preserved by $\phi_{c',c}$.
\end{proof}

\begin{figure}[htbp]
\begin{center}
\begin{tabular}{ccc}
\psscalebox{0.85}{\begin{pspicture}(0,5)(5,7)
\psarc(1,5){1}{0}{180}
\psarc(2.5,5){1.5}{0}{180}
\psarc(3.5,5){1.5}{0}{180}
\rput(0,5){$\bullet$}
\rput(1,5){$\bullet$}
\rput(2,5){$\bullet$}
\rput(3,5){$\bullet$}
\rput(4,5){$\bullet$}
\rput(5,5){$\bullet$}
\rput(0,4.7){\textrm{{\footnotesize 1}}}
\rput(1,4.7){\textrm{{\footnotesize 2}}}
\rput(2,4.7){\textrm{{\footnotesize 3}}}
\rput(3,4.7){\textrm{{\footnotesize 4}}}
\rput(4,4.7){\textrm{{\footnotesize 5}}}
\rput(5,4.7){\textrm{{\footnotesize 6}}}

\end{pspicture}} & ~~$\rightarrow$~~ & \psscalebox{0.85}{\begin{pspicture}(0,5)(5,7)
\psarc(1,5){1}{0}{180}
\psarc(2.5,5){1.5}{0}{180}
\psarc(3.5,5){1.5}{0}{180}
\rput(0,5){$\bullet$}
\rput(1,5){$\bullet$}
\rput(2,5){$\bullet$}
\rput(3,5){$\bullet$}
\rput(4,5){$\bullet$}
\rput(5,5){$\bullet$}
\rput(1.12,5.6){{\red{$\bullet$}}}
\rput(1.7, 5.73){{\red{$\bullet$}}}
\rput(0.9,6.02){{\red{$\bullet$}}}
\rput(1.62, 6.21){{\red{$\bullet$}}}
\rput(2.62,6.2){{\red{$\bullet$}}}
\rput(3.38, 6.2){{\red{$\bullet$}}}
\rput(2.62,6.5){{\red{$\bullet$}}}
\rput(3.38, 6.5){{\red{$\bullet$}}}
\rput(0,4.7){\textrm{{\footnotesize 1}}}
\rput(1,4.7){\textrm{{\footnotesize 2}}}
\rput(2,4.7){\textrm{{\footnotesize 3}}}
\rput(3,4.7){\textrm{{\footnotesize 4}}}
\rput(4,4.7){\textrm{{\footnotesize 5}}}
\rput(5,4.7){\textrm{{\footnotesize 6}}}

\psline[linecolor=red](1.12,5.6)(1.7, 5.73)
\psline[linecolor=red](0.9,6.02)(1.62, 6.21)
\psline[linecolor=red](2.62,6.2)(3.38, 6.2)
\psline[linecolor=red](2.62,6.5)(3.38, 6.5)
\end{pspicture}}\\
& & \\
& & $\downarrow$\\ 
\psscalebox{0.85}{\begin{pspicture}(0,5)(5,7)
\psarc(2.5,3.56){2.88}{30}{150}
\psarc(1.5,5){0.5}{0}{180}
\psarc(3,5){1}{0}{180}
\rput(0,5){$\bullet$}
\rput(1,5){$\bullet$}
\rput(2,5){$\bullet$}
\rput(3,5){$\bullet$}
\rput(4,5){$\bullet$}
\rput(5,5){$\bullet$}
\rput(0,4.7){\textrm{{\footnotesize 1}}}
\rput(1,4.7){\textrm{{\footnotesize 2}}}
\rput(2,4.7){\textrm{{\footnotesize 3}}}
\rput(3,4.7){\textrm{{\footnotesize 4}}}
\rput(4,4.7){\textrm{{\footnotesize 5}}}
\rput(5,4.7){\textrm{{\footnotesize 6}}}
\end{pspicture}} & ~~$\leftarrow$~~ & \psscalebox{0.85}{\begin{pspicture}(0,5)(5,7)
\psarc(1,5){1}{0}{45}
\psarc(1,5){1}{93}{180}
\psarc(2.5,5){1.5}{0}{52}
\psarc(2.5,5){1.5}{88}{127}
\psarc(2.5,5){1.5}{157}{180}
\psarc(3.5,5){1.5}{0}{93}
\psarc(3.5,5){1.5}{127}{180}
\rput(0,5){$\bullet$}
\rput(1,5){$\bullet$}
\rput(2,5){$\bullet$}
\rput(3,5){$\bullet$}
\rput(4,5){$\bullet$}
\rput(5,5){$\bullet$}
\rput(1.12,5.6){{\red{$\bullet$}}}
\rput(1.7, 5.73){{\red{$\bullet$}}}
\rput(0.9,6.02){{\red{$\bullet$}}}
\rput(1.62, 6.21){{\red{$\bullet$}}}
\rput(2.62,6.2){{\red{$\bullet$}}}
\rput(3.38, 6.2){{\red{$\bullet$}}}
\rput(2.62,6.5){{\red{$\bullet$}}}
\rput(3.38, 6.5){{\red{$\bullet$}}}
\rput(0,4.7){\textrm{{\footnotesize 1}}}
\rput(1,4.7){\textrm{{\footnotesize 2}}}
\rput(2,4.7){\textrm{{\footnotesize 3}}}
\rput(3,4.7){\textrm{{\footnotesize 4}}}
\rput(4,4.7){\textrm{{\footnotesize 5}}}
\rput(5,4.7){\textrm{{\footnotesize 6}}}
\psline[linecolor=red](1.12,5.6)(1.7, 5.73)
\psline[linecolor=red](0.9,6.02)(1.62, 6.21)
\psline[linecolor=red](2.62,6.2)(3.38, 6.2)
\psline[linecolor=red](2.62,6.5)(3.38, 6.5)
\end{pspicture}}\\
\end{tabular}
\end{center}
\caption{The bijection $\phi_{c',c}$ for $c'=(1,2,5,6,4,3)$, $x'=(2,5)(1,6,3)$, and $x=\phi_{c',c}(x)=(2,3,5)(1,6)$.}
\label{figure:process}
\end{figure}
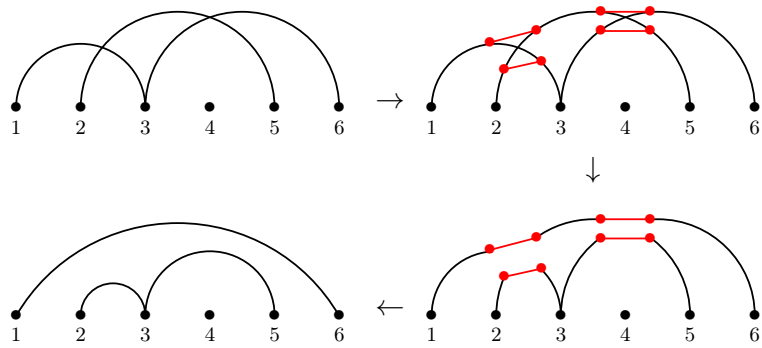

\begin{definition}
Consider the labeling of the circle corresponding to $c'$ and $(m_i, n_i)_i$ a collection of noncrossing segments between points with labels $n_i$ and $m_i$ such that if $i\neq j$, $\{n_i,m_i\}\cap\{n_j, m_j\}=0$. We say that $k\in[n+1]$ is \defn{exposed} to the segment $(m_j, n_j)$ if the segment joining the point labeled with $k$ to the point with label $m_j$ (equivalently $n_j$) does not cross any segment of the collection $(m_i, n_i)_i$. 
\end{definition}
\begin{proposition}\label{lem:ensb}
Let $x,y\in \NC(\W, c')$. If $x\neq y$, then $(D_{x}^{c'}, U_{x}^{c'})\neq(D_{y}^{c'}, U_{y}^{c'}).$ 
\end{proposition}
\begin{proof}
Given $(D:=D_{x}^{c'}, U:=U_{x}^{c'})$, the sets $I_{x}:=D\backslash(D\cap U)$ and $T_{x}:=U\backslash (D\cap U)$ contain the initial and terminal indices respectively. We first argue by induction on $|\Pol(x)|$ that there is a unique bijection $f:I_x\rightarrow T_x$ such that any two segments in $\{(a,f(a))\}_{a\in I_x}$ are noncrossing (when represented on the circle with the labeling of $c'$ from Subsection \ref{sub:orientation}).

To be noncrossing, the largest $i\in I_x$ must be joined to the element $t$ of $U$ which is bigger but as close as possible to $i$ on the circle (this is well defined since if $i\in R_{c'}$ (resp. $i\in L_{c'}$), then any element of $U$ bigger than $i$ is after $i$ (resp. before $i$) on the circle when going along it in clockwise order). Then $(i,t)$ is either a diagonal or an edge of a polygon $P\in\Pol(x)$---we call it the \defn{longest segment} of $P$, since $i$ is the minimal index of $P$ while $t$ is the maximal one.  Now if we divide the plane in two with the line extending the segment $(i,t)$, then all the points labeled with integers in $(I_x\backslash i)\cup(T_x\backslash t)$ lie in the same half-plane thanks to the properties of the orientation from Subsection \ref{sub:orientation}.  Now remove $P$ from $x$ to obtain $x'\in\NC(\W, c')$ for which the property holds by induction. But the collection of segments obtained from $x'$ clearly don't cross $(i, t)$ since they lie in the same half-plane defined by the line extending $(i,t)$. Hence if $x, y\in\NC(\W, c')$ were distinct but $(D_{x}^{c'}, U_{x}^{c'})=(D_y^{c'}, U_y^{c'})$, they would have the same family of longest segments of polygons. 

We must now check that an index $j$ of $D\cap U$ must lie in a single polygon (given here by its longest segment). If not, then $j$ is exposed to (at least) two noncrossing segments $(i_1, t_1)$ and $(i_2, t_2)$ of the family of longest segments such that $i_k<j<t_k$, $k=1,2$. But by Remark \ref{rmq:successifs}, the set $E_{j-1}=\{m\in\{1,\dots,n+1\}~|~m< j\}$ consists of points labeling vertices that are successive on the circle, and the same holds for the set $E'_{j+1}=\{m\in\{1,\dots,n+1\}~|~m> j\}$. Note that $i_k\in E_{j-1}$ and $t_k\in E'_{j+1}$. Call \emph{red points} the points labeled by elements of $E_{j-1}$ and \emph{blue points} the points labeled by ${E'}_{j+1}$. The two segments $(i_1,t_1)$ and $(i_2,t_2)$ join the blue part to the red part. Therefore, any index exposed to the two segments either lies in the blue part or in the red part (see Figure \ref{figure:bleurouge}). This is a contradiction since $j$ is assumed to be exposed to both segments but lies neither in $E_j$ nor in $E'_{j+1}$.
\end{proof}
\begin{figure}
\begin{center}
\psscalebox{0.85}{\begin{pspicture}(-2,-2)(2,2)
\pscircle(0,0){2}
\psdots(0,2)(2,0)(0,-2)(-2,0)(1.414,1.414)(-1.414,-1.414)(1.414,-1.414)(-1.414,1.414)(1.854,0.75)(0.75,1.854)(-1.854,0.75)(-0.75,1.854)(1.854,-0.75)(0.75,-1.854)(-1.854,-0.75)(-0.75,-1.854)
\psline(-0.75,1.854)(-1.854,-0.75)
\psline(2,0)(0,-2)
\rput(0,2){{\red{$\bullet$}}}
\rput(2,0){{\red{$\bullet$}}}
\rput(0,-2){{\blue{$\bullet$}}}
\rput(-2,0){$\bullet$}
\rput(1.414,1.414){{\red{$\bullet$}}}
\rput(1.414,-1.414){$\bullet$}
\rput(-1.414,1.414){$\bullet$}
\rput(-1.414,-1.414){{\blue{$\bullet$}}}
\rput(1.854,0.75){{\red{$\bullet$}}}
\rput(0.75,1.854){{\red{$\bullet$}}}
\rput(-1.854,-0.75){{\blue{$\bullet$}}}
\rput(-0.75,-1.854){{\blue{$\bullet$}}}
\rput(-1.854,0.75){$\bullet$}
\rput(-0.75,1.854){{\red{$\bullet$}}}
\rput(1.854,-0.75){$\bullet$}
\rput(0.75,-1.854){$\bullet$}
\rput(-1,1.95){\textrm{{\footnotesize $i_1$}}}
\rput(1.7,0){\textrm{{\footnotesize $i_2$}}}
\rput(-2.1,-0.75){\textrm{{\footnotesize $t_1$}}}
\rput(0,-1.7){\textrm{{\footnotesize $t_2$}}}
\end{pspicture}}
\end{center}
\caption{Figure for the proof of Proposition~\ref{lem:ensb}.  We fix $i_1,i_2,t_1,t_2$, and the segment between $i_2$ and $i_1$ that doesn't include $t_1$ or $t_2$, and similarly for $t_2$ and $t_1$.  All points between $i_2$ and $i_1$ are colored red, and all points between $t_2$ and $t_1$ are colored blue.  Then $j$ is neither red nor blue, and must therefore label a black point.} 
\label{figure:bleurouge}
\end{figure}
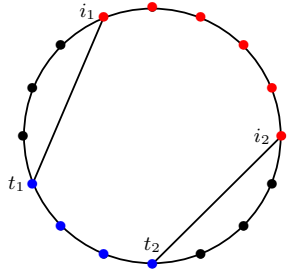
Let $y\in\NC(\W, c')$. We write $y_i^{c'}$ for the number of occurences of $s_i$ in the word $m_{y}^{c'}$.
\begin{proposition}\label{prop:nouvebij}
If $x:=\phi_{c',c}(y)$, then $y_i^{c'}=x_i$ for $i\in[n]$. That is, the number of copies of $s_i$ in $m_{y}^{c'}$ is the same as in $m_x=m_{\phi_{c',c}(y)}$.
\end{proposition}
\begin{proof}
We modify the word $m_{y}^{c'}$ at each step of the algorithm described in the proof of \ref{thm:bijccprime} so that the number of copies of each simple reflection does not change.
 Each step of the algorithm consists of replacing two crossing arcs $(i,k)(j,\ell)$ where $i<j<k<\ell$ by the two noncrossing arcs $(i,\ell)$ and $(j,k)$. Note that the way we represented the situation implies that the reflections $(i,k)$ and $(j,\ell)$ occur as syllables of $m_{y}^{c'}$. It suffices to replace the syllable $\mathbf{s_{[i,k]}}$ in $m_{y}^{c'}$ by $\mathbf{s_{[i,\ell]}}$ and $\mathbf{s_{[j,\ell]}}$ by $\mathbf{s_{[j,k]}}$ (or vice-versa; the order in which the syllables occur in the modified word is irrelevant, since we are interested in the \emph{number} of copies of each simple reflection in the word). This replacement evidently does not change the number.

If we allow the algorithm to run until termination, we obtain a word $m$ which is a product of the syllables of $m_x$ (but not necessarily in the right order, i.e.\ $m$ does not necessarily represent $x$).  Therefore $m$ has the same number of copies of each simple reflection as $m_x$, which concludes the proof.  Note that the order in which we choose to replace crossings does not affect the result, since it is geometrically clear that the arcs obtained at the end are the same for any order and represent exactly the syllables of $m_x$ (this last claim follows from the fact that the number of arcs starting or ending at a given point is constant during the algorithm).
\end{proof}
\subsubsection{Ordering noncrossing partitions for arbitrary Coxeter elements}
For $c'$ a Coxeter element, we write \[\phi_{c'}:\NC(\W, c')\rightarrow \Vn, ~x\mapsto (x_i^{c'})_{i=1}^n.\]

\noindent By Theorem~\ref{thm:bijccprime}, we have that the following diagram commutes.
\[
  \xymatrix{
   \NC(\W,c') \ar[rr]^{\phi_{c',c}} \ar[rrdd]_{\phi_{c'}} & & \NC(\W,c) \ar[dd]^{\phi_c} \\
   & &\\
     & & \Vn
  }
\]

The maps $\phi_{c',c}$ and $\phi_{c'}$ are bijections, but $\phi_c$ is actually an isomorphism of posets by Theorem~\ref{thm:caract}.  We can now induce the order of a distributive lattice on $\NC(\W,c')$, using the componentwise order on $\Vn$.  In general, this will \emph{not} be the same order as the Bruhat order on $\NC(\W,c')$.

\begin{example}
Figure \ref{figure:a3bis} compares the Hasse diagram of the induced order on $\NC(\mathfrak{S}_4,c')$ for $c'=s_2s_1s_3$ with Bruhat order.
\begin{figure}[htbp]
\begin{pspicture}(-2,0)(10,6)
\psline(7.8,0.2)(7.2,0.8)
\psline(7.2,1.2)(7.8,1.8)
\psline(8.8,1.2)(8.2,1.8)
\psline(8.8,2.2)(8.2,2.8)
\psline(8.8,4.2)(8.2,4.8)
\psline(6.8,2.2)(6.2,2.8)
\psline(6.8,3.8)(6.2,3.2)
\psline(9.2,2.2)(9.8,2.8)
\psline(9.8,3.2)(9.2,3.8)
\psline(8,0.3)(8,0.7)
\psline(8,2.3)(8,2.7)
\psline(8,5.3)(8,5.7)
\psline(7,1.3)(7,1.7)
\psline(9,1.3)(9,1.7)
\psline(8.2,0.2)(8.8,0.8)
\psline(7.8,1.2)(7.2,1.8)
\psline(7.8,3.2)(7.2,3.8)
\psline(8.2,1.2)(8.8,1.8)
\psline(8.2,3.2)(8.8,3.8)
\psline(7.2,2.2)(7.8,2.8)
\psline(7.2,4.2)(7.8,4.8)
\psline{->}(4.5,5.5)(5.5,5.5)
\rput(8,0){$e$}
\rput(7,1){$s_1$}
\rput(8,1){$s_2$}
\rput(9,1){$s_3$}
\rput(7,2){$s_2 s_1$}
\rput(8,2){$s_1 s_3$}
\rput(9,2){$s_2 s_3$}
\rput(8,3){$s_2 s_1 s_3$}
\rput(6,3){$s_2 s_1 s_2$}
\rput(10,3){$s_3s_2s_3$}
\rput(7,4){$s_2s_1s_2s_3$}
\rput(9,4){$s_3s_2s_3s_1$}
\rput(8,5){$s_3s_2s_1s_2s_3$}
\rput(8,6){{\red $s_2s_1s_2s_3s_2s_3$}}
\psline(1.8,0.2)(1.2,0.8)
\psline(1.2,1.2)(1.8,1.8)
\psline(2.8,1.2)(2.2,1.8)
\psline(2.8,2.2)(2.2,2.8)
\psline(2.8,4.2)(2.2,4.8)
\psline(0.8,2.2)(0.2,2.8)
\psline(2.8,3.8)(2.2,3.2)
\psline(3.2,2.2)(3.8,2.8)
\psline(3.8,3.2)(3.2,3.8)
\psline(2,0.3)(2,0.7)
\psline(2,2.3)(2,2.7)
\psline(1,1.3)(1,1.7)
\psline(0.8,4.2)(1.65,5.8)
\psline(3.2,4.2)(2.35,5.8)
\psline(2.2,0.2)(2.8,0.8)
\psline(1.8,1.2)(1.2,1.8)
\psline(1.8,3.2)(1.2,3.8)
\psline(2.2,1.2)(2.8,1.8)
\psline(0.2,3.2)(0.8,3.8)
\psline(1.2,2.2)(1.8,2.8)
\psline(1.2,4.2)(1.8,4.8)
\psline(3,1.3)(3,1.7)
\rput(2,0){$e$}
\rput(1,1){$s_1$}
\rput(2,1){$s_2$}
\rput(3,1){$s_3$}
\rput(1,2){$s_2 s_1$}
\rput(2,2){$s_1 s_3$}
\rput(3,2){$s_2 s_3$}
\rput(2,3){$s_2 s_1 s_3$}
\rput(0,3){$s_2 s_1 s_2$}
\rput(4,3){$s_3s_2s_3$}
\rput(1,4){$s_2s_1s_2s_3$}
\rput(3,4){$s_3s_2s_3s_1$}
\rput(2,5){{\red $s_2s_1s_3s_2$}}
\rput(2,6){$s_3s_2s_1s_2s_3$}
\end{pspicture}
\caption{On the left is the Bruhat order restricted to $\NC(\mathfrak{S}_4,c')$ for $c'=s_2 s_1 s_3$.  On the right is the distributive lattice structure on $\NC(\mathfrak{S}_4,c')$ induced by $\sf{V}_3.$  Elements on the right are specified using their standard forms.}
\label{figure:a3bis}
\end{figure}
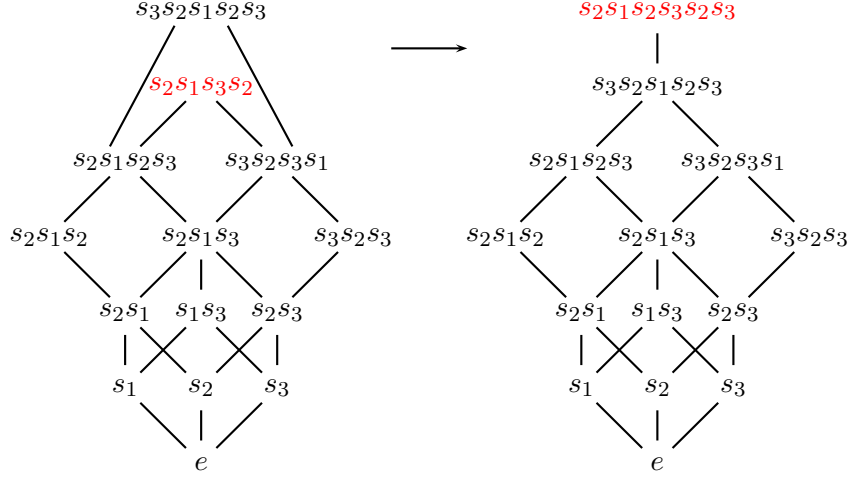
\end{example}
\begin{remark}
By applying Theorem \ref{thm:bijccprime} twice, we obtain bijections $\phi_{c',c''}$ for two arbitrary Coxeter elements $c',c''$.  These evidently generalize the bijection from Proposition \ref{prop:terminalinitial}.
\end{remark}

\subsection{Type $B_n$}
As previously mentioned, the lattice property fails in type $B_n$ (see Figure \ref{figure:b}). We present here analogous bijections to those given in type $A_n$, between nonnesting partitions and noncrossing partitions. The approach given here is the same as in Subsection \ref{sec:nn_diagonal}.

We fix the Weyl group of type $B_n$ as a permutation group on $[n]\cup-[n]$.  We write $((i,j))$ for
the cycle $(i,j)(-i,-j)$.  Then $B_n$ has $n$ simple
reflections~$\mathcal{S}:=\{s_i:=((i,i+1)) : 1 \leq i < n\} \cup \{s_n := (n,-n)\},$ $n$
simple roots $\{\alpha_i:=e_i-e_{i+1} : 1 \leq i < n\} \cup \{a_n:=e_n\}$, $n^2$
reflections~$\mathcal{T}:=\{((i,j)),((i,-j)) :1 \leq i < j \leq n\} \cup \{(-i,i) : 1 \leq i
\leq n\},$ positive roots~$\Phi^+:=\{e_i\pm e_j :1 \leq i < j \leq n\} \cup \{e_i :
1 \leq i \leq n\}$, and we will specialize to the Coxeter element $c:=s_0 s_1 \cdots
s_n = (1,2,\ldots,n,-1,-2,\ldots,-n)$.  Note that $A_{n-1}$ sits inside of $B_n$ as
the parabolic subgroup generated by $s_1,\ldots,s_{n-1}$.

As in type $A_n$, we define a map between nonnesting and noncrossing partitions and
prove that it is a bijection using the Kreweras complement.  We label the positive
roots $\Phi^+$ by $\sl(e_i)=s_i$, $\sl(e_i\pm e_j)=s_i$ for $|i-j|>1$, and
$\sl(e_i+e_{i+1}) = s_i \cdots s_n$. 

\begin{definition}
\label{def:bijB}
For type $B_n$, define $\bijDD: 2^{\Phi^+} \to S^*$ by
\[\bijDD(\nns) = \left( \prod_{\substack{\alpha \in \nns \\
\hgt(\alpha)=1,3,5,\ldots}} s(\alpha) \right) \left( \prod_{\substack{\alpha \in
\nns \\ \hgt(\alpha) = 2,4,6,\ldots}} \sl(\alpha) \right)^{-1},\] where if
$\hgt(\alpha)=\hgt(\beta)$, $\sl(\alpha)=s_i$, $\sl(\beta)=s_j$ or
$\sl(\beta)=s_j\cdots s_n$, and $i<j$, then $\alpha$ comes before $\beta$ in the
product.  Define $\bijD: \NN(B_n) \to B_n$ by evaluating $\bijDD$ restricted to
$\NN(B_n)$ as an element of $B_n$.
\end{definition}

\begin{remark}
By labeling the root poset by $e_i-e_j \mapsto s_i$, $e_i \mapsto s_i$, and $e_i+e_j
\mapsto s_{n-(j-i-1)}$, a similar product map as in Remark~\ref{rem:sort} gives a
bijection between $\NN(B_n)$ and the $c$-sortable elements of type $B_n$, and maps
the number of positive roots in the nonnesting partition to the length of the
corresponding $c$-sortable element. See~\cite{Stump} for a more comprehensive
treatment of this bijection.
\end{remark}


Given a nonnesting partition $\nns$, let $k$ be minimal so that $\alpha_{k+1} \not
\in \nns$. The \defn{initial part} $\Initial(\nns)$ is the nonnesting partition in
either $A_k$ or $B_n$ (if $k=n$) that coincides with $\nns$ restricted to
$\alpha_1,\alpha_2,\ldots,\alpha_k$. The \defn{final part} $\Final(\nns)$ is the
nonnesting partition in $B_{n-k}$ that coincides with $\nns$ restricted to
$\alpha_{k+1},\alpha_{k+2},\ldots,\alpha_{n}$. The following lemma is immediate
from the definition.

\begin{definition}
\label{def:krewmoveb}
Define the action $\KrewNN_c$ on $\nns \in \NN(B_n)$ by replacing each root $e_i-e_j \in
\Initial(\nns)$ by $e_i-e_{j-1}$, each root $e_i \in \Initial(\nns)$ by $e_i-e_n$,
each root $e_i+e_j \in \Initial(\nns)$ for $|i-j|>1$ by $e_i+e_{j+1}$ (where
$e_{n+1}=0$), and each root $e_i+e_{i+1} \in \Initial(\nns)$ by $e_i+e_{i+2}$ and
$e_{i+1}+e_{i+2}$. Each root $e_i-e_j \in \Final(\nns)$ is replaced by
$e_{i-1}-e_j$, each root $e_i \in \Final(\nns)$ by $e_{i-1}$ ($e_n$ also adds an
$e_{n-1}+e_n$), and roots $e_i+e_j$ are replaced by $e_{i-1}+e_{j}$. For each root
$e_i+e_{i+2}$ added, $e_i+e_{i+1}$ is also added, and finally all simple roots in
the support of $\Final(\nns)$ $\alpha_{k+1},\ldots,\alpha_n$ are added.
\end{definition}

It is not hard to check that $\KrewNN_c$ is invertible on $\NN(B_n)$.

\begin{remark}
We note that $|\NN(B_n)|=|\J([n]\times[n])|$. There are several bijections between
these two sets (\cite{Stembridge,Proctor,Reiner}). The Kreweras complement on
$\NN(B_n)$ produces the same orbit structure as the action of $c$ on the parabolic
quotient $A_{2n}^J$ for $J=\{s_n\}$, whose elements have inversion sets in natural
bijection with the order ideals in $\J([n]\times[n])$. This action of $c$ on the
parabolic quotient has been successfully modeled using toggles in~\cite{RushShi}.
\end{remark}

For $c$ a Coxeter element of type $B_n$, we write $\Pc$ for the set of elements of $B_n$ which are below $c$ in absolute order.
\begin{theorem}
\label{them:typeb}
The map $\bijD$ is a bijection from $\NN(B_n)$ to $\Pc$.
\end{theorem}
\begin{proof}
The proof is essentially the same as in type $A_n$ (Theorem \ref{them:typea}).

\begin{enumerate}
        \item If the nonnesting partition $\nns$ does not contain every simple root, as
parabolic subgroups of type $B$ are of type $A$ or of type $B$, we conclude the
statement by induction on rank.

        \item Otherwise, we have a nonnesting partition $\nns$ containing every simple
root. It is immediate from Definition~\ref{def:krewmoveb} that $\bijD$ is equivariant
with respect to $\KrewNN_c$ on such a nonnesting partition, since $\KrewNN_c$
removes one element from every diagonal until reaching a root $e_i+e_{i+1}$ (which
then contains the desired remaining simple reflections): \[\bijD(\KrewNN_c(\nns)) =
\bijD(\nns)^{-1} c=\Krew_c(\bijD(\nns)).\]  Then we can apply $\KrewNN_c$ until we
obtain a nonnesting partition that does not contain the simple root $\alpha_n$. 
Let the number of applications required be $k$---by Definition~\ref{def:krewmoveb},
$k\leq n$. Since $\Krew_c$ is invertible on $\Pc$, we conclude by induction
on rank that
\[\bijD(\nns)=\Krew^{-k}_c(\Krew^k_c(\bijD(\nns)))=\Krew^{-k}_c(\bijD(\KrewNN^k_c(\nns)))\]
is a bijection.
\end{enumerate}

Since $|\NN(B_n)|=|\Pc|=\binom{2n}{n}$, we conclude the result.
\end{proof}

\begin{corollary}
        Let $\nns \in \NN(B_n)$. The rank of $\bijDD(\nns)$ in absolute order is equal to 
        \[\lt(\nns):=|\{\alpha \in \nns : \hgt(\alpha) \text{ odd} \text{ and }  \alpha
\neq e_{n-1-2i}+e_{n-2i}\}| - |\{\alpha \in \nns : \hgt(\alpha) \text{ even}\}|.\] 
In particular, there are $\binom{n}{k}^2$ elements of $\NN(B_n)$ with
$\lt(\nns)=k$.
\end{corollary}

\end{document}